\documentclass[11pt]{amsart}

\usepackage[T1]{fontenc}
\usepackage[utf8]{inputenc}
\usepackage{textcomp}
\usepackage{lmodern}
\usepackage{microtype}
\usepackage{mathabx}

\usepackage{amssymb}
\usepackage{amsthm}
\usepackage{amscd}

\usepackage{hyperref}

\usepackage{mathscinet}
\usepackage[giveninits=true,doi=false,url=false,sortcites,backend=biber,backref=true]{biblatex}
\usepackage{bussproofs}

\newtheorem{theorem}{Theorem}[section]

\newtheorem{lemma}[theorem]{Lemma}

\newtheorem{cor}[theorem]{Corollary}

\theoremstyle{definition}
\newtheorem{definition}[theorem]{Definition}

 \def \dom{\operatorname{dom}}

\mathchardef\mhyphen="2D

\allowdisplaybreaks[2]

\begin{document}

\title{Proofs that Modify Proofs\\{\MakeTitlecase{Work in Progress}}}
\author{Henry Towsner}
\date{\today}
\thanks{Partially supported by NSF grant DMS-2054379}
\address {Department of Mathematics, University of Pennsylvania, 209 South 33rd Street, Philadelphia, PA 19104-6395, USA}
\email{htowsner@math.upenn.edu}
\urladdr{\url{http://www.math.upenn.edu/~htowsner}}

\begin{abstract}
  In this paper we give an ordinal analysis of the theory of second order arithmetic.   We do this by working with proof trees---that is, ``deductions'' which may not be well-founded. Working in a suitable theory, we are able to represent functions on proof trees as yet further proof trees satisfying a suitable analog of well-foundedness. Iterating this process allows us to represent higher order functions as well: since functions on proof trees are just proof trees themselves, these functions can easily be extended to act on proof trees which are themselves understood as functions.  The corresponding system of ordinals parallels this, using higher order collapsing function.
\end{abstract}

\maketitle

\section{Introduction}

In this paper, we give an ordinal analysis of second-order arithmetic.\footnote{The current draft represents an in-progress version of this work. The ordinal analysis is currently incomplete: we give a ``qualitative'' ordinal analysis, working in a strong theory and proving the existence of certain ordinal bounds without calculating them explicitly. A future version of this paper will include an explicit ordinal notation system. In addition, certain proofs are currently only sketched and a great deal of exposition is missing.}  We follow the conventional approach in infinitary ordinal analysis as introduced in \cite{MR45673} (see, for instance, \cite{Pohlers2008-hh} for a modern exposition): we begin with a finitary proof in second-order arithmetic, transfer it to a proof in an infinitary system, and show that we can transform any such proof into a cut-free proof, and further that there can be no cut-free proof of $0=1$. This allows us to conclude that second-order arithmetic.

Such a proof can, of course, only be given in a sufficiently strong system. This strength appears in our proof through the use of ordinal bounds: what we show is that our infinitary system has a cut-free proof with a certain kind of ordinal bounds. We then use the well-foundedness of the system of ordinal bounds to conclude that there is no such proof of $0=1$. The goal of ordinal analysis is, roughly, to calibrate the strength of the theory by identifying the ordinal needed to carry out such a proof \cite{MR2275588,MR1720577,MR1943301}.

An ordinal analysis of second-order arithmetic, using more purely finitary techniques, was also recently announced by Arai \cite{arai2024ordinal}. Our approach here was developed independently and uses rather different methods.

The main challenge is how to deal with second-order quantifiers. Buchholz's $\Omega$ rule \cite{MR1943302,MR969597} gives an approach to eliminating cuts between a $\Pi^1_1$ and $\Sigma^1_1$ statement. In this approach, the way we prove\footnote{As is customary, we use a one-sided sequent calculus, so by a proof of a finite set of formulas, we mean a proof that at least one of these formulas holds.} $\exists^2X\,phi(X),\Sigma$ is to give a function which transforms a proof of $\neg\phi(X),\Delta$ (with suitable restrictions) into a proof of $\Delta\Sigma$.

A proof of $\exists^2X\,\phi(X),\Sigma$ in a finitary system might derive this from some $\phi(A),\Sigma$. Suppose we have already, inductively, translated the proof of $\phi(A),\Sigma$ into our infinitary system. Then we could hope to produce such a function by taking a proof of $\neg\phi(X),\Delta$, replacing $X$ with $A$ everywhere, and then using the Cut rule to combine our proof of $\phi(A),\Sigma$ with our proof of $\neg\phi(A),\Delta$.

To generalize this to a $\Sigma^1_2$ formula, however, we face a new difficulty: we would like a proof of $\exists^2X\forall^2Y\,\phi(X,Y),\Sigma$ to be a function which acts on proofs of $\exists^2Y\,\neg\phi(X,Y),\Delta$. When we try to produce such a function, however, we want to be able to convert a proof of $\exists^2Y\,\neg\phi(X,Y),\Delta$ into a proof of $\exists^2Y\,\neg\phi(A,Y),\Delta$. It is not immediately clear how to do this: we are given a function whose domain is proofs of $\phi(X,Y),\Delta'$, but we need to produce a function whose domain is proofs of $\phi(A,Y),\Delta'$.

Our solution is to use the fact that the functions we need are not arbitrary functions.  The premises of the Buchholz $\Omega$ are, essentially, the \emph{graph} of the function we are using. Here, we use a modified $\Omega$ rule, which we call the $\Omega^\flat$ rule, which takes as its premise a certain description of the algorithm being used. In particular, these algorithms are always ``continuous'' \cite{Mints1978}, in the sense that the bottom part of the output is determined by the bottom part of the input.

The crucial insight is that these algorithms can be seen as proof-trees (that is, potentially ill-founded proofs): we have nodes in our tree (that is, we have a special inference rule) where we can look at a position in the input deduction, and the branches of our tree (that is, the premises of this inference rule) are the possible rules that the input deduction could have at that position.

Such a proof-tree is not well-founded; the right property to ask for is that it transforms well-founded proofs to well-founded proofs. This is represented by assigning a kind of higher-order ordinal bound to the tree.\footnote{Although we do not explore this here, we expect that these higher-order ordinal bounds are essentially dilators\cite{MR656793}.}

So a proof of $\exists^2Y\,\neg\phi(X,Y),\Delta$ in our infinitary proof system will be an \emph{ill-founded proof-tree} which describes the process by which a proof of $\phi(X,Y),\Delta'$ is transformed into a proof of $\Delta'\Delta$. We are then able to transform this proof-tree into a proof of $\exists^2Y\,\neg\phi(A,Y),\Delta$, which is precisely what we need to do.

Because we are representing our functions on proof-trees by new proof-trees, we can simply iterate this idea. Along the way, we need more and more complicated ordinal bounds---bounds which guarantee that we transform higher-order ordinal bounds into higher-order ordinal bounds, and so on.\footnote{Again, we expect a fundamental connection to the Girard's ptykes \cite{MR903244} which is not explored here.}

\section{Second Order Arithmetic}

\subsection{Language}

Our ultimate interest is the theory of second order arithmetic, which is in the language $\mathcal{L}_2$. We take this language to be defined as follows.
\begin{definition}
  The symbols of $\mathcal{L}_{2}$ consist of:
  \begin{itemize}
  \item an infinite set $Var=\{x,y,z,\ldots\}$ for first-order variables,
  \item a constant symbol $0$,
  \item a unary function symbol $\mathrm{S}$,
  \item for every $n$-ary primitive recursive relation $R$, an $n$-ary predicate symbol $\dot{R}$,
  \item an infinite set $Var_2=\{X,Y,Z,\ldots\}$ for second-order variables,
  \item the unary connective $\neg$,
  \item two binary connectives, $\wedge$ and $\vee$,
  \item two first-order quantifiers, $\forall$ and $\exists$,
  \item two second-order quantifiers $\forall^2$ and $\exists^2$,
  \end{itemize}
  where all symbols are distinct.

  The terms are given by:
  \begin{itemize}
  \item Each first-order variable is a term,
  \item $0$ is a term,
  \item if $t$ is a term, $St$ is a term.
  \end{itemize}

  The atomic formulas are defined by:
  \begin{itemize}
  \item whenever $\dot{R}$ is a symbol for an $n$-ary primitive recursive relation and $t_1,\ldots,t_n$ are terms, there is an atomic formula $\dot{R}t_1\cdots t_n$,
  \item whenever $X$ is a second order variable and $t$ is a term, $Xt$ is a atomic formula.
  \end{itemize}

  We define formulas as follows:
  \begin{itemize}
  \item whenever $\eta$ is an atomic formula, $\eta$ and $\neg\eta$ are formulas, which we call \emph{literals},
  \item whenever $\phi$ and $\psi$ are formulas, $\phi\wedge\psi$ and $\phi\vee\psi$ are formulas,
  \item whenever $\phi$ is a formula and $x$ is a first-order variable, $\forall x\,\phi$ and $\exists x\,\phi$ are formulas,
  \item whenever $\phi$ is a formula and $X$ is a second-order variable, $\forall^2 X\,\phi$ and $\exists^2 X\,\phi$ are formulas.
  \end{itemize}
\end{definition}

\begin{definition}
  We define the negation of a formula, ${\sim}\phi$, inductively by:
  \begin{itemize}
  \item if $\sigma$ is an atomic formula, ${\sim}\sigma$ is $\neg\sigma$ and ${\sim}\neg\sigma$ is $\sigma$,
  \item ${\sim}(\phi\wedge\psi)$ is $({\sim}\phi)\vee({\sim}\psi)$,
  \item ${\sim}(\phi\vee\psi)$ is $({\sim}\phi)\wedge({\sim}\psi)$,
  \item ${\sim}\forall x\,\phi$ is $\exists x\,{\sim}\phi$,
  \item ${\sim}\exists x\,\phi$ is $\forall x\,{\sim}\phi$,
  \item ${\sim}\forall^2 X\,\phi$ is $\exists^2 X\,{\sim}\phi$,
  \item ${\sim}\exists^2 X\,\phi$ is $\forall^2 X\,{\sim}\phi$.
  \end{itemize}
\end{definition}

\subsection{Inference Rules}

\begin{definition}
  A \emph{sequent} is a set of formulas.
\end{definition}
This definition, together with our choices below, builds the usual structural rules (contraction, weakening, and exchange) into the definition of a deduction.

\begin{definition}
  An \emph{inference rule} is a tuple $\mathcal{R}=(I,\Delta,\{\Delta_\iota\}_{\iota\in I},\{E_\iota\}_{\iota\in I})$ where $I$ is a set, $\Delta$ is a sequent, $\{\Delta_\iota\}_{\iota\in I}$ is a set of sequents indexed by $I$, and each $E_\iota$ is a subset of $Var\cup Var_2$.
\end{definition}
We usually write inference rules using the notation

\AxiomC{$\cdots$}
\AxiomC{$\Delta_\iota$}
\AxiomC{$\cdots\ (\iota\in I)$}
\LeftLabel{$\mathcal{R}$}
\RightLabel{$!\{E_\iota\}!$}
\TrinaryInfC{$\Delta$}
\DisplayProof

When $|I|=1$, we simply write $!E!$ on the right hand side. When $|E|=1$, we write $!x!$ or $!X!$ in place of $!\{x\}!$ or $!\{X\}!$, respectively. When $E=\emptyset$, we omit the right-hand label entirely.

The label $\mathcal{R}$ is the name of the inference rule, the set $I$ is the \emph{premises of $\mathcal{R}$}, the sequent $\Delta$ is the \emph{conclusion sequent}, the sequent $\Delta_\iota$ is the \emph{premise sequent at $\iota$}, and the sets $E_\iota$ are the \emph{eigenvariables}. When $\mathcal{R}$ is an inference rule, we write $|\mathcal{R}|$ for its premises and $\Delta(\mathcal{R})$, $\Delta_\iota(\mathcal{R})$ for the corresponding sequents. We write $Eig_\iota(\mathcal{R})$ for the set of eigenvariables, again omitting $\iota$ when $|I|=1$.

It is our assumption that the sets $I$, across all rules ever considered, are disjoint---that is, given $\iota\in |\mathcal{R}|$, we can recover the rule $\mathcal{R}$ from $\iota$, and we denote it $\mathcal{R}(\iota)$. We will often equivocate notationally, writing rules whose sets of premises are $\mathbb{N}$, $\{L,R\}$, or $\{\top\}$. We always take it as given that the set of premises is \emph{really} some set of pairs from $\{t\}\times\mathbb{N}$ or $\{t\}\times\{L,R\}$ or the like, where $t$ is a label for the rule.

For reasons that will become clear, our perspective on deductions emphasizes, not that they are built from axioms in a well-founded way, but that they are ``read'' from the root up. In particular, our basic object is a \emph{proof-tree}, the co-well-founded analog of a deduction.

\begin{definition}
  A \emph{theory} is a set of inference rules.

  When $\mathfrak{T}$ is a theory, a \emph{proof-tree in $\mathfrak{T}$} is a function $d$ where:
  \begin{itemize}
  \item the range of $d$ is $\mathcal{T}$,
  \item $\langle\rangle\in\dom(d)$,
  \item if $\sigma\in\dom(d)$ and $\iota\in |d(\sigma)|$ then $\sigma\iota\in\dom(d)$.
  \end{itemize}

  For any $\sigma\in\dom(d)$, we define $\Gamma(d,\sigma)$ to be the set of $\phi$ such that there is some $\tau$ with $\sigma\tau\in\dom(d)$, $\phi\in\Delta(d(\sigma\tau))$, and for all $\tau'\iota\sqsubseteq\tau$, $\phi\not\in\Delta_\iota(d(\sigma\tau'))$. We write $\Gamma(d)$ for $\Gamma(d,\langle\rangle)$.
\end{definition}

As usual, $\Gamma(d)$ is called the \emph{conclusion} of $d$.

\begin{definition}
  When $d$ is a proof-tree and $\sigma\in\dom(d)$, we write $d_\sigma$ for the \emph{sub-tree}  given by $d_\sigma(\tau)=d(\sigma\tau)$.
\end{definition}

It is sometimes natural to consider some $\sigma$ which could be in the domain of a deduction without referring to the specific deduction it comes from.

\begin{definition}
  When $\mathfrak{T}$ is a theory, a $\mathfrak{T}$-sequence is a finite sequence $\sigma$ such that each element of $\sigma$ is some $\iota$ with $\mathcal{R}(\iota)\in\mathfrak{T}$.  We define $\Delta(\langle\rangle)=\emptyset$ and $\Delta(\sigma\iota)=\Delta_\iota(\mathcal{R}(\iota))$.

  When $\sigma$ is a $\mathfrak{T}$-sequence, we define $\Gamma^{\leftarrow}(\sigma)$ inductively by $\Gamma^{\leftarrow}(\langle\rangle)=\emptyset$ and $\Gamma^{\leftarrow}(\sigma\iota)=(\Gamma^{\leftarrow}(\langle\rangle)\setminus\Delta(\mathcal{R}(\iota)))\cup\Delta_\iota(\mathcal{R}(\iota))$.
\end{definition}
We may think of $\Gamma^{\leftarrow}(\sigma)$ as representing the ``formulas permitted at the top of $\sigma$'': if $\sigma\in\dom(d)$ then $\Gamma(d,\sigma)\subseteq\Gamma(d,\langle\rangle)\cup\Gamma^\leftarrow(\sigma)$. (It is calculated, of course, by following the rules for calculating the conclusion ``backwards''.)

\begin{definition}
  A proof-tree $d$ is \emph{valid} if, for every $\sigma\in\dom(d)$ and every $\iota\in|d(\sigma)$, no element of $Eig_\iota(d(\sigma))$ appears free in any formula in $\Gamma(d,\sigma\iota)$.
\end{definition}
Throughout this paper, all proof-trees will be assumed to be valid.

\begin{definition}
  A proof-tree $d$ is a \emph{deduction} if $d$ is well-founded.
\end{definition}
Equivalently, $d$ is well-founded if there is an assignment of ordinals, $o^d:\dom(\sigma)\rightarrow Ord$ such that $o^d(\tau)<o^d(\sigma)$ whenever $\sigma\sqsubsetneq\tau$.

\subsection{Second Order Arithmetic}

The theory $PA_2$ contains the following rules.

\AxiomC{}
\LeftLabel{True$_{\dot{R}t_1\cdots t_n}$}
\UnaryInfC{$\eta$}
\DisplayProof
where $t_1,\cdots, t_n$ are closed, $\eta\in\{\dot{R}t_1\cdots t_n,\neg\dot{R}t_1\cdots t_n\}$, and $\eta$ is true.

\bigskip

\AxiomC{}
\LeftLabel{Ax$_{\{\eta,{\sim}\eta\}}$}
\UnaryInfC{$\eta,\neg \eta$}
\DisplayProof
where $\eta$ is a literal.

\bigskip

\AxiomC{$\phi$}
\AxiomC{$\psi$}
\LeftLabel{I$\wedge_{\phi\wedge\psi}$}
\BinaryInfC{$\phi\wedge\psi$}
\DisplayProof

\bigskip

\AxiomC{$\phi$}
\LeftLabel{I$\vee^L_{\phi\vee\psi}$}
\UnaryInfC{$\phi\vee\psi$}
\DisplayProof
\quad\quad\quad
\AxiomC{$\psi$}
\LeftLabel{I$\vee^R_{\phi\vee\psi}$}
\UnaryInfC{$\phi\vee\psi$}
\DisplayProof

\bigskip

\AxiomC{$\phi(y)$}
\LeftLabel{I$\forall^y_{\forall x\phi}$}
\RightLabel{$!y!$, where $y$ is substitutable for $x$ in $\phi$}
\UnaryInfC{$\forall x\,\phi$}
\DisplayProof

\bigskip

\AxiomC{$\phi(t)$}
\LeftLabel{I$\exists^t_{\exists x\,\phi}$}
\RightLabel{where $t$ is substitutable for $x$ in $\phi$}
\UnaryInfC{$\exists x\,\phi$}
\DisplayProof

\bigskip

\AxiomC{$\phi(Y)$}
\LeftLabel{I$\forall^Y_{\forall^2 X\phi}$}
\RightLabel{$!Y!$, where $Y$ is substitutable for $X$ in $\phi$}
\UnaryInfC{$\forall^2 X\,\phi$}
\DisplayProof

\bigskip

\AxiomC{$\phi[X\mapsto \psi]$}
\LeftLabel{I$\exists^\psi_{\exists^2X\,\phi}$}
\RightLabel{where $\psi$ is substitutable for $X$ in $\phi$}
\UnaryInfC{$\exists^2X\,\phi$}
\DisplayProof

  \bigskip

\AxiomC{}
\LeftLabel{Ind$^t_{\phi(x)}$}
\RightLabel{where $t$ is substitutable for $x$ in $\phi$}
\UnaryInfC{${\sim}\phi(0), \exists x\,(\phi(x)\wedge {\sim}\phi(Sx)), \phi(t)$}
\DisplayProof

\bigskip

\AxiomC{$\phi$}
\AxiomC{${\sim}\phi$}
\LeftLabel{Cut$_{\phi}$}
\BinaryInfC{$\emptyset$}
\DisplayProof

\smallskip

When one of these rules has a single premise, we call it $\top$. The two premises of the I$\wedge$ are $L$ and $R$, for left and right. The two premises of the Cut rule are $\top$ and $\bot$.

We write $PA_2\vdash\Sigma$ if there is a valid deduction $d$ in $PA_2$ with $\Gamma(d)\subseteq\Sigma$. We write $PA_{2,<0}$ for the fragment of $PA_2$ omitting Cut, and similarly write $PA_{2,<0}\vdash\Sigma$ if there is a valid deduction $d$ in $PA_{2,<0}$ with $\Gamma(d)\subseteq\Sigma$. An ordinal analysis, such as the one we will give below, shows, among other things, that $PA_2\not\vdash\emptyset$.

\section{Local Functions on Proof-Trees}

In this section, we describe certain kinds of functions from proof-trees to proof-trees, namely those which are defined ``one rule at a time'' starting from the root.

Our motivating example is, of course, the cut-elimination operation itself. It is an observation going back to \cite{Mints1978,MR1943302} that (at least when we have the Rep rule \AxiomC{$\emptyset$}\UnaryInfC{$\emptyset$}\DisplayProof), if $d$ is a deduction and $d'$ is the result of applying the cut-elimination operators to $d$, then the bottom part of $d'$ is determined by a suitable bottom part of $d$.

Here we introduce some machinery, building on \cite{MR1943302} and \cite{towsner:MR2499713}, for representing this operation as a proof-tree itself.

\subsection{Encoding Local Functions}

To deal with proof-trees that represent functions, we want to introduce a new kind of ``formula'', which we call a \emph{tag}.  When $F$ is a proof-tree, a tag in $\Gamma(F)$ indicates that $F$ should be ``interpreted as a function''. (Because weakening is built into our system, note that we can then interpret \emph{any} proof-tree as a function: if there is no tag in $\Gamma(d)$ then $d$ will represent a constant function.)

\begin{definition}
  A \emph{tag} is a triple $(\Theta_0,\epsilon,\Theta)$ where $\Theta_0\subseteq\Theta$ are sequents and $\epsilon$ is a $\mathfrak{T}$-sequence for some $\mathfrak{T}$. We call $\Theta_0$ the \emph{root} of the tag.

  When $t=(\Theta_0,\epsilon,\Theta)$, we write $\mathsf{r}(t)$ for $\Theta_0$, $\epsilon(t)$ for $\epsilon$, and $\Theta(t)$ for $\Theta$.

\end{definition}
The tag we expect to see in $\Gamma(F)$ is $(\Theta_0,\langle\rangle,\Theta_0)$. At some internal position $\Gamma(F,\sigma)$, we might instead see some $(\Theta_0,\epsilon,\Theta)$; this essentially means that, at position $\sigma$ in $F$, we have ``looked at'' the part of the input up to $\epsilon$. The sequent $\Theta_0$ will be removed from the conclusion of the input.

The link between the sequent $\Theta_0$ and the theory $\mathfrak{T}$ will be built into our proof system.

We need to expand our sequents a bit.
\begin{definition}
  An \emph{extended sequent} is a pairs of sets $\Delta;\Upsilon$ where $\Delta$ is a sequent and $\Upsilon$ is a set of tags. We write $\cdot_{\mathsf{t}}$ for the tag part of a sequent, so $(\Delta;\Upsilon)_{\mathsf{t}}=\Upsilon$.

  When $\Delta;\Upsilon$ is an extended sequent and $(\mathfrak{T},\Theta_0,\epsilon,\Theta)$ is a tag, we take $\Delta;\Upsilon\setminus\{(\Theta_0,\epsilon,\Theta)\}$ to mean $\Delta;(\Upsilon\setminus\{(\Theta_0,\epsilon,\Theta')\mid\Theta\subseteq\Theta'\})$.
\end{definition}
That is, we abuse notation to understand removing a tag $(\Theta_0,\epsilon,\Theta)$ from an extended sequent to remove all tags of the form $(\Theta_0,\epsilon,\Theta')$ with $\Theta\subseteq\Theta'$.

We apply this convention, in particular, to the definition of $\Gamma(d,\sigma)$: $(\Theta_0,\epsilon,\Theta')\in\Gamma(d,\sigma)$ if there is a $\tau$ with $t\in\Delta(\tau)$ and for all $\tau'\iota\sqsubseteq\tau$, there is no $(\Theta_0,\epsilon,\Theta)\in\Delta_{\iota}(d(\sigma\tau'))$ with $\Theta\subseteq\Theta'$.

Corresponding to this, we extend the definition of a tag to include extended sequents. (More precisely, we have to define extended sequents and tags by a simultaneous induction: new tags give us new extended sequents, which give us yet more tags.)

The following rule is essential to representing local functions as proof-trees.

\AxiomC{$\Delta(\mathcal{R})\setminus\Theta; \{(\Theta_0,\epsilon\iota,\Theta)\mid \iota\in|\mathcal{R}|\}$}
\AxiomC{$(\mathcal{R}\in\mathfrak{T})$}
\LeftLabel{Read$_{\mathfrak{T},\Theta_0,\epsilon,\Theta}$}
\BinaryInfC{$\Delta(\epsilon)\setminus\Theta; (\Theta_0,\epsilon,\Theta)$}
\DisplayProof

Note that this rule branches over rules from $\mathfrak{T}$. Reading from the bottom up (as we often read our proof-trees), we interpret this rule to say: ``look at the input corresponding to the tag $\Theta_0$, in the position $\epsilon$; if the rule here is $\mathcal{R}$, follow the $\mathcal{R}$-branch of this rule''.

We need one slightly ad hoc restriction on proof-trees to support our notion of ordinal bounds later.
\begin{definition}
  When $d$ is a proof-tree, a \emph{consecutive Read} is a pair $(\sigma,\iota)$ with $\sigma\in\dom(d)$, $\sigma\iota\in\dom(d)$ so that $d(\sigma)=Read_{\mathfrak{T},t}$ and $d(\sigma\iota)=Read_{\mathfrak{T},t'}$ where $t$ and $t'$ share the same root.

  We say $d$ is \emph{free of consecutive Reads} if $d$ does not contain any consecutive Reads.
\end{definition}
We will always want our proof-trees to be free of consecutive Reads. (In practice, we can accomplish this by including the Rep rule in our system and using Rep rules to separate otherwise consecutive Reads.) Therefore, for the remainder of this paper, we assume proof-trees are free of consecutive Reads.

\begin{definition}
  A \emph{locally defined function from $\mathfrak{T}$ to $\mathfrak{T}'$} is a sequent $\Theta_0$ together with a proof-tree $F$ in $\mathfrak{T}'+\{\mathrm{Read}_{\mathfrak{T},\Theta_0,\epsilon,\Theta}\mid \Theta_0\subseteq\Theta\}$.

  More generally, a \emph{locally defined function from $\mathfrak{T}_1,\ldots,\mathfrak{T}_k$ in $\mathfrak{T}'$} is a sequence of sequents $\Theta_0^1,\ldots,\Theta_0^k)$ together with a proof-tree $F$ in $\mathfrak{T}'+\bigcup_{i\leq k}\{\mathrm{Read}_{\mathfrak{T}^i,\Theta^i_0,\epsilon,\Theta}\mid \Theta^i_0\subseteq\Theta\}$.
\end{definition}
We will see that we can interpret a locally defined function as a function from proof-trees in $\mathfrak{T}$ to proof-trees in $\mathfrak{T}'$. (Forbidding consecutive Reads is necessary to rule out a degenerate case.)

Note that a locally defined function is essentially just a proof-tree which we choose---via our choice of the root tag---to view as a function. We typically simply say $F$ is a locally defined function and write $\Theta_F$ for the associated root.

\begin{definition}
  Let $d$ be a proof-tree in $\mathfrak{T}$ and $F$ a locally defined function from $\mathfrak{T}$ to $\mathfrak{T}'$. We define a proof-tree $\bar F(d)$ in $\mathfrak{T}'$ along with auxiliary functions $h_0,h$ from $\dom(\bar F(d))$ to $\dom(F)$:
  \begin{itemize}
  \item $h_0(\langle\rangle)=\langle\rangle$,
  \item if $F(h_0(\sigma))=\mathrm{Read}_{\mathfrak{T},t}$ where $\mathsf{r}(t)=\Theta_F$ then $h(\sigma)=h_0(\sigma)d(\epsilon_t)$; otherwise $h(\sigma)=h_0(\sigma)$,
  \item $\bar F(d)(\sigma)=F(h(\sigma))$ and, for all $\iota\in|F(h(\sigma))|$, $h_0(\sigma\iota)=h(\sigma)\iota$,
  \end{itemize}
\end{definition}
That is, we copy over the proof-tree $F$, except that at $\mathrm{Read}$ rules we delete the rule and use $d$ to decide which branch to take. (If we had consecutive Read rules, we could have an infinite branch consisting entirely of Read rules which looks at the input but never places any rules, which would lead $\bar F(d)$ to be undefined.)

The key point is that the conclusion of $\bar F(d)$ is the combination of the conclusions of $d$ and the conclusion of $F$ in precisely the way we expect.
\begin{theorem}
  $\Gamma(\bar F(d))\subseteq (\Gamma(d)\setminus\Theta_F) \cup (\Gamma(F)\setminus\{(\Theta_F,\langle\rangle,\Theta_F)\}) $.
\end{theorem}
\begin{proof}
  Suppose $\phi\in\Gamma(\bar F(d))$, so there is some $\sigma\in \dom(\bar F(d))$ so that $\phi\in\Delta(\bar F(d)(\sigma))$ and $\phi\not\in\Delta_\iota(\bar F(d)(\tau))$ for any $\tau\iota\sqsubseteq\sigma$. 

  If $\phi\in\Gamma(F)$, we are done, so suppose not. We have $\phi\in \Delta(F(h(\sigma)))$. Choose $\sigma'\sqsubseteq h(\sigma)$ minimal so that $\phi\in\Delta(F(\sigma'))$. Since $\phi$ is not removed at any rule of $\bar F(d)$ below $\sigma$, $\phi$ must be removed at some Read$_{\mathfrak{T},\Theta_F,\epsilon,\Theta}$ inference below $\sigma'$, so $\phi\in\Delta(\mathcal{R})\setminus\Theta$ where $\mathcal{R}$ is $d(\epsilon)$ for some $\epsilon\in\dom(d)$.

  Suppose there is some $\epsilon'\iota\sqsubseteq\epsilon$ with $\phi\in\Delta_\iota(d(\epsilon'))$. Then there is some corresponding $\tau\sqsubseteq \sigma'$ so that $F(\tau)$ is Read$_{\mathfrak{T},\Theta_F,\epsilon',\Theta'}$ for some $\Theta'\subseteq\Theta$. Therefore $\phi\not\in\Theta'$, so $\phi\in\Delta(F(\tau))$, contradicting the minimality of $\sigma'$.

  Therefore $\phi\in\Gamma(d)$.
\end{proof}

\subsection{Extensions of Locally Defined Function}

One of the crucial features of locally defined functions is that there is a natural way to extend them to apply to larger theories, simply by acting as the identity on new rules. Given a locally defined function $F$ from $\mathfrak{T}$ and some other theory $\mathfrak{T}^+$ (presumably extending $\mathfrak{T}$, though we don't need this assumption), we can try to extend $F$ to $\mathfrak{T}^+$ by assuming $F$ ``does nothing'' on rules in $\mathfrak{T}^+\setminus\mathfrak{T}$.

\begin{definition}\label{def:lifting}
  Given a locally defined function $F$ on $\mathfrak{T}$ with tag $t$, we define the \emph{lift} of $F$ to $\mathfrak{T}^+$, $F^\uparrow$, by induction as follows, along with functions $h:\dom(F^\uparrow)\rightarrow\dom(F)$ and $\pi_\sigma: I_\ast(s_t(h(\sigma)))\rightarrow I_\ast(s_t(\sigma))$:
  \begin{itemize}
  \item $h(\langle\rangle)=\langle\rangle$,
  \item $\pi_{\langle\rangle}(\langle\rangle)=\langle\rangle$,
  \item if $F(h(\sigma))$ is a Read$_{\mathfrak{T},\Theta_F,\epsilon,\Theta}$ rule then:    
    \begin{itemize}
    \item $F^\uparrow(\sigma)$ is Read$_{\mathfrak{T}^+,\Theta_F,\pi_\sigma(\epsilon),\Theta}$,
    \item for each $\mathcal{R}$ in $\mathfrak{T}$,
      \begin{itemize}
      \item  $h(\sigma\mathcal{R})=h(\sigma)\mathcal{R}$,
      \item for each $\epsilon'\in I_\ast(s_t(h(\sigma)\mathcal{R}))$, if $\epsilon'\in I_\ast(s_t(h(\sigma)))$ then $\pi_{\sigma\mathcal{R}}(\epsilon')=\pi_\sigma(\epsilon')$, and if $\epsilon'$ is $\epsilon\iota$ for some $\iota\in|\mathcal{R}|$ then $\pi_{\sigma\mathcal{R}}(\epsilon')=\pi_\sigma(\epsilon)\iota$,
      \end{itemize}
    \item for each $\mathcal{R}$ in $\mathfrak{T}^+\setminus\mathfrak{T}$ other than a Read rule:
      \begin{itemize}
      \item  $F^\uparrow(\sigma\mathcal{R})=\mathcal{R}$,
      \item for $\iota\in|\mathcal{R}|$, $h(\sigma\mathcal{R}\iota)=h(\sigma\mathcal{R})=h(\sigma)$,
      \item for each $\epsilon'\in I_\ast(s_t(h(\sigma\mathcal{R}\iota)))=I_\ast(s_t(h(\sigma)))$, if $\epsilon'\neq\epsilon$ then $\pi_{\sigma\mathcal{R}\iota}(\epsilon')=\pi_\sigma(\epsilon')$, and $\pi_{\sigma\mathcal{R}\iota}(\epsilon)=\pi_\sigma(\epsilon)\iota$,
      \end{itemize}
    \item for each Read$_{\mathfrak{T}',\Theta',\epsilon',\Theta''}$ in $\mathfrak{T}^+\setminus\mathfrak{T}$:
      \begin{itemize}
      \item $F^\uparrow(\sigma\mathcal{R})$ is Read$_{\mathfrak{T}',\Theta',\epsilon',\Theta''\cup\Theta}$,
      \item for $\iota\in|\mathcal{R}|$, $h(\sigma\mathcal{R}\iota)=h(\sigma\mathcal{R})=h(\sigma)$,
      \item for each $\epsilon'\in I_\ast(s_t(h(\sigma\mathcal{R}\iota)))=I_\ast(s_t(h(\sigma)))$, if $\epsilon'\neq\epsilon$ then $\pi_{\sigma\mathcal{R}\iota}(\epsilon')=\pi_\sigma(\epsilon')$, and $\pi_{\sigma\mathcal{R}\iota}(\epsilon)=\pi_\sigma(\epsilon)\iota$,
      \end{itemize}
    \end{itemize}
  \item If $F(h(\sigma))$ is any other rule then
    \begin{itemize}
    \item $F^\uparrow(\sigma)=F(h(\sigma))$,
    \item $h(\sigma\iota)=h(\sigma)\iota$ for all $\iota\in |F(h(\sigma))|$,
    \item $\pi_{\sigma\iota}=\pi_\sigma$ for all $\iota\in|F(h(\sigma))|$,
    \end{itemize}
  \end{itemize}
\end{definition}

\begin{definition}
  Let $\mathcal{F}$ be a set of formulas. We say $\mathfrak{T}^+$ is \emph{rule-by-rule conservative over $\mathfrak{T}$ for $\mathcal{F}$} if:
  \begin{itemize}
  \item whenever $\mathcal{R}\in\mathfrak{T}^+\setminus\mathfrak{T}$ is not a Read inference, $\Delta(\mathcal{R})\cap\mathcal{F}=\emptyset$,
  \item if $\mathcal{R}\in\mathfrak{T}^+\setminus\mathfrak{T}$ is a Read inference then $\Delta(\mathcal{R})_{\mathsf{t}}\cap\mathcal{F}=\emptyset$.
  \end{itemize}
\end{definition}

\begin{lemma}\label{operator_extension}
  Let $F$ be a locally defined function on $\mathfrak{T}$ and suppose that for every Read$_{\mathfrak{T},\Theta_F,\epsilon,\Theta}$ rule appearing in $F$, $\mathfrak{T}^+$ is rule-by-rule conservative over $\mathfrak{T}$ for $\Theta$.

  Then $\Gamma(F^\uparrow)\subseteq\Gamma(F)$.
\end{lemma}
\begin{proof}
  Consider any $\sigma$ and any $\phi\in\Delta(F^\uparrow(\sigma))$ and suppose $\phi\not\in\Gamma(F)$. We will show that there is a $\tau\iota\sqsubseteq\sigma$ with $\phi\in\Delta_\iota(F^\uparrow(\tau))$.

  Suppose $\Delta(F(h(\sigma)))=\Delta(F^\uparrow(\sigma))$. Since $\phi\not\in\Gamma(F)$, there is some $\tau\iota\sqsubseteq h(\sigma)$ with $\phi\in\Delta_\iota(F(\tau))$; but there is a $\tau'\iota\sqsubseteq\sigma$ with $h(\tau')=\tau$ and $F^\uparrow(\tau')=F(\tau)$, and therefore $\phi\in\Delta_\iota(F^\uparrow(\tau'))$.

  Otherwise $F^\uparrow(\sigma)$ must be $\mathcal{R}$ for some $\mathcal{R}\in\mathfrak{T}^+\setminus\mathfrak{T}$, and therefore $\sigma=\sigma'\mathcal{R}$ where $F^\uparrow(\sigma')$ is a Read$_{\mathfrak{T}^+,\Theta_F,\epsilon,\Theta}$ inference. If $\mathcal{R}$ is not a Read inference then we have $\Delta(F^\uparrow(\sigma))=\Delta(\mathcal{R})$ and therefore $\Delta(F^{\uparrow}(\sigma))$ is disjoint from $\Theta$; but this means $\phi\in\Delta_{\mathcal{R}}(\sigma')$ as promised.

  If $\mathcal{R}$ is a Read inference in $\mathfrak{T}^+\setminus\mathfrak{T}$ then we have $\Delta(\mathcal{R})\cap\Theta=\emptyset$, and therefore again $\phi\in\Delta_{\mathcal{R}}(\sigma')$.
\end{proof}

\section{Infinitary Theory}

\subsection{Language and Theories}

We need to modify our language by requiring that free second-order variables have a ``level''. The level will be a natural number which roughly represents the depth of nested parameters in the construction of sets. Bound variables will not have a level.

\begin{definition}
  The language $\mathcal{L}_{2}^\infty$ is the language $\mathcal{L}_{2}$ with the following modification:
  \begin{itemize}
  \item the set $Var_2$ of second order variables is partitioned into infinitely many sets  $Var_2=\bigcup_{\ell\in\mathbb{N}\cup\{\ast\}} Var_2^\ell$ where each $Var_2^\ell=\{X^\ell,Y^\ell,Z^\ell,\ldots\}$ is infinite.
  \end{itemize}

\end{definition}

The second-order variables labeled by $\ast$ are the unleveled ones---that is, the variables that appear bound. We omit the $\ast$ level when writing variables, so free second-order variables should always have a level indicated as a superscript, while bound second-order variables have no superscript.

For formal reasons, we will need the Rep rule:

\AxiomC{$\emptyset$}
\LeftLabel{Rep}
\UnaryInfC{$\emptyset$}
\DisplayProof

\bigskip

As usual, our infinitary system will replace the I$\forall$ rule with the $\omega$ rule:
\AxiomC{$\cdots\phi(n)\cdots$}
\AxiomC{$(n\in\mathbb{N})$}
\LeftLabel{$\omega_{\forall x\,\phi}$}
\BinaryInfC{$\forall x\,\phi$}
\DisplayProof

Finally, we introduce the distinctive rule for our infinitary system, a modified form of the Buchholz $\Omega$ rule, which will let us include locally defined functions inside our deductions.

When $\mathfrak{T}$ is a theory, the $\Omega^\flat$ rule\footnote{We think of this as ``flattening'' Buchholz's version of the $\Omega$ rule: it has a single premise which is essentially an algorithm describing how to calculate a function, where the premises of the $\Omega$ rule essentially encode the graph of the function.} is

\AxiomC{$(\{{\sim}\phi(Y^\ell)\},\langle\rangle,\{{\sim}\phi(Y^\ell)\})$}
\LeftLabel{$\Omega^\flat_{Y^\ell,\exists^2X\,\phi}$}
\UnaryInfC{$\exists^2X\,\phi$}
\DisplayProof

As usual \cite{MR1943302}\cite{MR3490922}, we need a partner rule that combines this rule with a Cut inference:

\AxiomC{${\sim}\phi(Z^\ell)$}
\AxiomC{$(\{{\sim}\phi(Y^\ell)\},\langle\rangle,\{{\sim}\phi(Y^\ell)\})$}
\RightLabel{$!\{Z^\ell\},\emptyset!$}
\LeftLabel{Cut$\Omega^\flat_{ Z^\ell, Y^\ell,\exists^2X\,\phi}$}
\BinaryInfC{$\emptyset$}
\DisplayProof

The eigenvariables here indicate, of course, that $Z^\ell$ is an eigenvalue in the left premise and there are no eigenvariables in the right premise.

These are the only rules which can remove a tag from a premise without adding one in the conclusion. These allows us to have proof-trees which infer formulas from the existence of functions. We call the unique premise of $\Omega^\flat$ the $\bot$ premise\footnote{This is inconsistent with the way we have named the premises of other single premise rules, but reflects that we think of the premise of an $\Omega^\flat$ rule as an essentially negative thing, since it is a function on proofs of the positive thing.}, and then Cut$\Omega^\flat$ has two premises, $\top$ and $\bot$.

We introduce two measures of complexity on formulas. Roughly speaking, the \emph{depth} will measure the nesting of second-order quantifiers while the \emph{level} will measure how the formula depends on parameters. In particular, the formula $\exists^2X\,\phi(X,\forall^2Y\,\psi(X,Y))$ will have depth $2$ and level $0$, while the formula $\exists^2X\,\phi(X,\forall^2Y\,\psi(Y))$ will have depth $1$ and level also $1$---the first formula involves a nested quantifier, while the second involves a single quantifier relative to a parameter (the completed formula $\forall^2Y\,\psi(Y)$). (More precisely, we are interested in the quantities $dp(\phi,\{X\})$ and $lvl(\phi,\{X\})$ when considering a formula $\exists^2X\,\phi$.)

\begin{definition}
  If $\mathcal{S}$ is a set of second-order variables, $dp(\phi,\mathcal{S})$ is defined inductively by:
  \begin{itemize}
  \item when $\phi$ is a literal, $dp(\phi,\mathcal{S})=0$,
  \item when $\phi$ is $\psi_0\wedge\psi_1$ or $\psi_0\vee\psi_1$, $dp(\phi,\mathcal{S})=max\{dp(\psi_0,\mathcal{S}), dp(\psi_1,\mathcal{S})\}$,
  \item when $\phi$ is $\forall x\,\psi$ or $\exists x\,\psi$, $dp(\phi,\mathcal{S})=dp(\psi,\mathcal{S})$,
  \item if any variable in $\mathcal{S}$ is free in $\psi$ then $dp(\forall^2Y\,\psi,\mathcal{S})=dp(\exists^2Y\,\psi,\mathcal{S})=dp(\psi,\mathcal{S}\cup\{Y\})+1$,
  \item if no variable in $\mathcal{S}$ is free in $\psi$ then $dp(\forall^2Y\,\psi,\mathcal{S})=dp(\exists^2Y\,\psi,\mathcal{S})=0$.
  \end{itemize}

  $lvl(\phi,\mathcal{S})$ is defined inductively by:
    \begin{itemize}
    \item when $\phi$ is $X^\ell n$ or $\neg X^\ell n$ then $lvl(\phi,\mathcal{S})=\ell+1$,
    \item when $\phi$ is any other literal, $lvl(\phi,\mathcal{S})=0$,
  \item when $\phi$ is $\psi_0\wedge\psi_1$ or $\psi_0\vee\psi_1$, $lvl(\phi,\mathcal{S})=max\{lvl(\psi_0,\mathcal{S}), lvl(\psi_1,\mathcal{S})\}$,
  \item when $\phi$ is $\forall x\,\psi$ or $\exists x\,\psi$, $lvl(\phi,\mathcal{S})=lvl(\psi,\mathcal{S})$,
  \item if any variable in $\mathcal{S}$ is free in $\psi$ then $lvl(\forall^2Y\,\psi,\mathcal{S})=lvl(\exists^2Y\,\psi,\mathcal{S})=lvl(\psi,\mathcal{S}\cup\{Y\})$,
  \item if no variable in $\mathcal{S}$ is free in $\psi$ then $lvl(\forall^2Y\,\psi,\mathcal{S})=lvl(\exists^2Y\,\psi,\mathcal{S})=lvl(\psi,\emptyset)+1$.
  \end{itemize}
  We write $lvl(\phi)$ for $lvl(\phi,\emptyset)$.

Given a formula of the form $Q^2X\,\phi$, we write $comp(Q^2X\,\phi)$ for the pair $(dp(\phi,\{X\}),lvl(\phi,\{X\}))$, which we call the \emph{complexity} of $Q^2X\,\phi$.  We compare complexities in reverse lexicographic order: $(m',\ell')<(m,\ell)$ if $\ell'<\ell$ or $\ell'=\ell$ and $m'<m$.
\end{definition}

We need to build a suitable hierarchy of theories, because using these rules requires specifying suitable theories to use as the domains of our $\Omega^\flat$ and Cut$\Omega^\flat$ rules.

In order to keep our bounds suitably uniform, we want to work in theories with a global bound on the complexity. The theory $PA_{2,<0}^{\infty,N,L}$ will only include $\Omega^\flat$ and Cut$\Omega^\flat$ rules introducing formulas with complexity less than $(N,L)$. We do not attempt here to match the strength of these theories to natural finitary theories (in particular, while the theories $PA_{2,<0}^{\infty,N,L}$ as $L$ goes to infinity probably a strength similar to that of $\Pi^1_{N}\mhyphen CA_0$, we have not attempted to calibrate them to make this exact). (As motivation for the system of levels, note that the systems $PA_{2,<0}^{\infty,1,L}$ have roughly $L$ iterations of $\Pi^1_1$-comprehension---that is, $PA_{2,<0}^{\infty,1,L}$ is similar in strength to $ID_L$, though again, we have not attempted to make this exact.)

\begin{definition}
  \begin{itemize}
  \item the base theory $PA_{2,<0}^{\infty,N,L,-}$ consists of the following rules, subject to the constraints that first-order variables do not appear free, second-order variables of level $\ast$ do not appear free, and there are no quantifiers over variables $Y^\ell$:
    \begin{itemize}
    \item Rep,
    \item $\mathrm{True}_\eta$ where $\eta$ is a true literal,
    \item Ax, I$\wedge$, I$\vee$, $\omega$, I$\exists$, 
    \item I$\forall^{Y^{lvl(\forall^2X\,\phi)}}_{\forall^2X\,\phi}$,
    \item $\Omega^\flat_{Y^\ell,\exists^2X\,\phi}$ where $\ell$ is $lvl(\phi,\{X\})$,
    \end{itemize}
  \item $C^{\infty,m,\ell}$ is the theory containing only rules Cut$\Omega^\flat_{Z^{\ell},Y^{\ell},\exists^2X\,\phi}$ where $(m,\ell)=\mathrm{comp}(\exists^2X\,\phi)$,
  \item $R^{\infty,N,L,m,\ell}$ is the theory containing only rules Read$_{PA_{2,<0}^{\infty,N,L,m,\ell,-},{\sim}\phi(Y^{\ell}),\epsilon,\Theta}$  where $(m,\ell)=\mathrm{comp}(\exists^2X\,\phi)$,
  \item  $PA_{2,<0}^{\infty,N,L,m,\ell,-}$ is $PA_{2,<0}^{\infty,N,L,-}+\bigcup_{(m',\ell')<(m,\ell)}C^{\infty,m',\ell'}+\bigcup_{m'<m,\ell'\leq L}R^{\infty,N,L,m',\ell'}$,
  \item $PA_{2,<0}^{\infty,N,L,m,\ell}$ is $PA_{2,<0}^{\infty,n,m,\ell,-}+\bigcup_{m'\leq N,\ell'<\ell}R^{\infty,N,L,m',\ell'}$.
  \item $PA_{2,<r}^{\infty,N,L,m,\ell}$ is $PA_{2,<0}^{\infty,n,m,\ell}$ plus Cut rules over all formulas of rank $<r$,
  \item $PA_{2,<r}^{\infty,n<L}$ is $\bigcup_{m< N,\ell< L}PA_{2,<r}^{\infty,N,L,m,\ell}$.
  \end{itemize}
\end{definition}

The difference between $PA_{2,<0}^{\infty,N,L,m,\ell,-}$ and $PA_{2,<0}^{\infty,N,L,m,\ell}$ is the presence of Read rules for low level ($\ell'< \ell$) but high depth ($m'\geq m$) theories. We will never really prove things in $PA_{2,<0}^{\infty,N,L,m,\ell,-}$; rather, we will lift functions from the domain $PA_{2,<0}^{\infty,N,L,m,\ell,-}$ to the domain $PA_{2,<0}^{\infty,N,L,m,\ell}$. However this restriction will play a crucial role in our ordinal bounds.

\begin{lemma}
  For any $\phi$, $X$, letting $(m,\ell)=comp(\exists^2X\,\phi)$, for any $n\geq m$ and $L\geq \ell$ there is a locally defined function $\mathrm{Id}^{\phi}$ in $PA_{2,<0}^{\infty,N,L}$ so that $\Gamma(\mathrm{Id}^{\phi})\subseteq\phi;(\phi,\langle\rangle,\phi)$ where 
\end{lemma}
\begin{proof}
We construct Id$^\phi$ by induction on $\sigma$. We set Id$^\phi(\langle\rangle)=Read_{PA_{2,<0}^{\infty,N,L,m,\ell,-},\phi,\langle\rangle,\phi}$.

  Given any $\sigma$ with Id$^\phi(\sigma)=Read_{PA_{2,<0}^{\infty,N,L, m,\ell,-},\phi,\epsilon,\phi}$, for each $\mathcal{R}$, we define Id$^\phi(\sigma\mathcal{R})=\mathcal{R}$, and for each $\iota\in|\mathcal{R}|$, we define Id$^\phi(\sigma\mathcal{R}\iota)=Read_{PA_{2,<0}^{\infty,N,L, m,\ell,-},\phi,\epsilon\iota,\phi}$.
\end{proof}

\begin{lemma}
  For any closed formula $\phi$ there are $N,L$ large enough so there is a proof-tree $d_\phi$ in PA$^{\infty,N,L}_{2,<0}$ with $\Gamma(d_\phi)\subseteq\{\phi,{\sim}\phi\}$.
\end{lemma}
\begin{proof}
  By induction on $\phi$. When $\phi$ is a literal, we take $d_\phi(\langle\rangle)=Ax_{\{\phi,{\sim}\phi\}}$.

  If $\phi$ is constructed by $\wedge,\vee,\forall,\exists$, this is standard.

  Suppose $\phi$ is $\forall^2 X\,\psi$. Then, choosing a fresh variable $Y^{\ell}$, we may take $d_{\phi}$ to be

  \AxiomC{Id$^{\psi}$}
  \noLine
  \UnaryInfC{$\vdots$}
  \noLine
  \UnaryInfC{$\psi(Y^{\ell}); (\psi(Y^\ell),\langle\rangle,\psi(Y^\ell)))$}
  \RightLabel{$\Omega^\flat_{Y^{\ell},\exists^2X\,{\sim}\psi}$}
  \UnaryInfC{$\psi(Y^{\ell}),\exists^2X\,{\sim}\psi$}
  \RightLabel{I$\forall^{Y^\ell}_{\forall^2X\,\psi}$}
  \UnaryInfC{$\forall^2X\,\psi,\exists^2X\,{\sim}\psi$}
  \DisplayProof

  where the values of $m,\ell$ are determined by $\psi$.

  The $\exists^2 X\,\psi$ case is identical, using ${\sim}\psi$ in place of $\psi$.
\end{proof}

\subsection{Embedding}

We show that deductions in $PA_2$ can be embedded in $PA_{2,<r}^{\infty,N,L}$ for suitable $N,L,r$. Most of the work here is embedding the I$\exists^2$ rule, which has to be replaced by a suitable $\Omega^\flat$ rule. (The I$\forall$ rule is, of course, replaced by an $\omega$ rule, but this step is by now standard.)

Here, as well as in the cut-elimination functions we define later, there is a technical matter of handling possible conflicts between names of second order variables. We ignore this issue here: we assume all second-order variables have already been suitably renamed.

\begin{theorem}
  For any formulas $\phi(X^\ell)$ and $\psi$, taking $(m,\ell)=comp(\exists^2X\,\phi(X))$, for all sufficiently large $N,L$ there is a locally defined function $F^{\phi,X^\ell\mapsto\psi}$ in $PA_2^{\infty,N,L,m,\ell}$ with $\Gamma(F)\subseteq \phi(\psi);(\phi(X^{\ell}),\langle\rangle,\phi(X^\ell))$.
\end{theorem}

The construction is exactly what one would expect: we proceed through the input proof, replacing $X^\ell$ with $\psi$ everywhere it appears. When we encounter some $\Omega^\flat$ rule introducing a formula $\exists^2Y\,\phi'(X^\ell,Y)$, we replace it with an $\Omega^\flat$ rule introducing $\exists^2Y\,\phi'(\psi,Y)$; this may mean changing the theory we branch over---replacing $X^\ell$ with $\psi$ may change the level (but, importantly, not the depth) of the formula.

We have to keep track of this change in theories, since we will then encounter Read rules whose root tag comes from that $\Omega^\flat$ rule, and we need to replace these with Read rules with the modified root tag.
\begin{proof}
  We define $F^{\phi,X^\ell\mapsto\psi}$ by induction on $\sigma$. We set $F^{\phi,X^\ell\mapsto\psi}(\langle\rangle)=Read_{PA_{2,<0}^{\infty,N,L,m,\ell,-},\phi(X^\ell),\langle\rangle,\phi(X^\ell)}$ and $\epsilon_{\langle\rangle}=\langle\rangle$.

Consider some $\sigma$ where $F^{\phi,X^\ell\mapsto\psi}(\sigma)$ is $Read_{PA_{2,<0}^{\infty,N,L,m,\ell,-},\phi(X^\ell),\epsilon,\Theta}$. We will maintain inductively that whenever $(\mathfrak{T},\Lambda_0,\zeta,\Lambda)$ is a tag in $\Theta$, we have an associated tag $(\mathfrak{T}',\Lambda'_0,\zeta',\Lambda')$ so that there is a $\tau\iota\sqsubseteq\sigma$ with $(\mathfrak{T}',\Lambda'_0,\zeta',\Lambda')\in \Delta_\iota(F^{\phi,X^\ell\mapsto\psi}(\tau))$, and we will have $\Lambda'\subseteq\Lambda[X^\ell\mapsto\psi]\cup\Theta$.

 For each $\mathcal{R}$ in $PA_{2,<0}^{\infty,N,L,m,\ell,-}$, we will define $F^{\phi,X^\ell\mapsto\psi}(\sigma\mathcal{R})$.

  If $\Delta(\mathcal{R})\cap\Theta=\emptyset$, we may take $F^{\phi,X^\ell\mapsto\psi}(\sigma\mathcal{R})=\mathcal{R}$ and $F^{\phi,X^\ell\mapsto\psi}(\sigma\mathcal{R}\iota)=Read_{PA_{2,<0}^{\infty,N,L,m,\ell,-},\phi(X^\ell),\epsilon\iota,\Theta}$. (Note that, if $\mathcal{R}$ is Cut$\Omega^\flat$ then $X^\ell$ cannot appear free in the cut formula.)

  If $\Delta(\mathcal{R})\cap\Theta\neq\emptyset$ and $\mathcal{R}$ is anything other than $Ax$ or a Read, then $\Delta(\mathcal{R})=\{\theta\}$ for some formula $\theta$; then there is a corresponding rule $\mathcal{R}'$ which introduces $\theta[X\mapsto\psi]$, so we take $F^{\phi,X^\ell\mapsto\psi}(\sigma\mathcal{R})=\mathcal{R}'$ and $F^{\phi,X^\ell\mapsto\psi}(\sigma\mathcal{R}\iota)=Read_{PA_{2,<0}^{\infty,N,L,m,\ell},\phi(X^\ell),\epsilon\iota,\Theta\cup\Delta_\iota(\mathcal{R})}$.  In the case where $\mathcal{R}$ is $\Omega^\flat$, note that we do satisfy the requirement on the presence of a replacement tag; the theory $\mathfrak{T}$ in this case has the form $PA_{2,<0}^{\infty,N,L,m',\ell',-}$, and we are replacing it with some $PA_{2,<0}^{\infty,N,L,m',\ell''}$---that is, the level might change (and indeed, might increase), but substituting $\psi$ for $X^\ell$ does not change the depth.
  
  If $\mathcal{R}$ is Ax$_{\{X^\ell n,\neg X^\ell n\}}$ then we may replace it with $d_{\psi(n)}$ (that is, we set $F^{\phi,X^\ell\mapsto\psi}(\sigma\mathcal{R}\tau)=d_{\psi(n)}(\tau)$ for all $\tau$).

  The remaining case is that $\mathcal{R}$ is a Read$_{\mathfrak{T},\Lambda_0,\zeta,\Lambda'}$ rule whose conclusion overlaps with $\Theta$. Here $\mathfrak{T}$ must be $PA^{\infty,N,L,m',\ell',-}_{2,<0}$ for some $m'<m$ and $\ell'\leq \ell$. If $\Delta(\mathcal{R})\cap\Theta\subseteq\Delta(\zeta)$, we set $F^{\phi,X^\ell\mapsto\psi}(\sigma\mathcal{R})=Read_{\mathfrak{T}, \Lambda_0,\zeta,\Lambda'\cup\Theta}$ and, as usual, set $F^{\phi,X^\ell\mapsto\psi}(\sigma\mathcal{R}\iota)=Read_{PA_{2,<0}^{\infty,N,L,m',\ell'},\phi(X^\ell),\epsilon\iota,\Theta\cup\Delta_\iota(\mathcal{R})}$.

  The really new case is when $(\mathfrak{T},\Lambda_0,\zeta,\Lambda)\in\Theta$. In this case we replaced this tag with a new tag, $(\mathfrak{T}',\Lambda'_0,\zeta',\Lambda')$, and so we set $F^{\phi,X^\ell\mapsto\psi}(\sigma\mathcal{R})=Read_{\mathfrak{T}',\Lambda'_0,\zeta',\Lambda'}$.

  Since $\mathfrak{T}'$ need not be $\mathfrak{T}$, the premises do not line up, so we have to consider how to handle $\sigma\mathcal{R}\mathcal{R}'$. Consider some $\mathcal{R}'\in\mathfrak{T}'$. If $\Delta(\mathcal{R}')\cap \Lambda[X^\ell\mapsto\psi]=\emptyset$ then it is simpler to ignore what the input would have done in the $\mathcal{R}'$ branch; we instead apply the same argument to $\mathcal{R}'$ that we applied to $\mathcal{R}$, making all extensions to the $\epsilon \mathrm{Rep}$ branch of the input. (Note that this recurses---$\mathcal{R}'$ could itself be another Read rule---but only finitely many times, since $m'$ drops at each step.)

  So suppose $\Delta(\mathcal{R}')\cap \Lambda[X^\ell\mapsto\psi]\neq\emptyset$. In this case there is a rule $\mathcal{R}^\leftarrow\in\mathfrak{T}$ so that $\Delta(\mathcal{R}^\leftarrow)[X^\ell\mapsto\psi]=\Delta(\mathcal{R}')$. Then we take $F^{\phi,X^\ell\mapsto\psi}(\sigma\mathcal{R}\mathcal{R}')=Read_{PA_{2,<0}^{\infty,N,L,m,\ell,-},\phi(X^\ell),\epsilon\mathcal{R}^\leftarrow,\Theta\cup\Delta_\iota(\mathcal{R})}$.
\end{proof}

\begin{theorem}\label{thm:embedding}
  Let $d$ be a valid deduction in $PA_2$ so that every formula in $\Gamma(d)$ has first-order free variables contained in $x_1,\ldots,x_a$ and second-order free variables contained in $X_1,\ldots,X_b$.

  Then for sufficiently large $n,L$ and $r$, for every $n_1,\ldots,n_a,\ell_1,\ldots,\ell_b$, there is a valid proof-tree $d^\infty_{n_1,\ldots,n_a}$ in $PA^{\infty,n,L}_{2,<r}$ with $\Gamma(d^\infty_{n_1,\ldots,n_a,\ell_1,\ldots,\ell_b})\subseteq \Gamma(d)[x_i\mapsto n_i,X_j\mapsto X_j^{\ell_j}]$.
\end{theorem}
\begin{proof}
  By induction on $d$. We consider cases on $d(\langle\rangle)$.

  If $d(\langle\rangle)$ is True, it is already a deduction of the desired form, and if $d(\langle\rangle)$ is Ax$_{\sigma}$ then we may take $d^{\infty}_{n_1,\ldots,n_a,\ell_1,\ldots,\ell_b}(\langle\rangle)$ to be $Ax_{\sigma[x_i\mapsto n_i,X_j\mapsto X_j^\ell]}$.

  If $d(\langle\rangle)$ is I$\wedge$, I$\vee$, I$\exists$, or $Cut$ the claim follows immediately from the inductive hypothesis.

  If $d(\langle\rangle)$ is I$\forall^y_{\forall x\,\phi}$ then, for each $n_{a+1}$, we may apply the inductive hypothesis to obtain a deduction $d^\infty_{n_1,\ldots,n_a,n_{a+1},\ell_1,\ldots,\ell_b}$ with $\Gamma(d^\infty_{n_1,\ldots,n_a,n_{a+1},\ell_1,\ldots,\ell_b})\subseteq \Gamma(d)[x_i\mapsto n_i,X_j\mapsto X_j^{\ell_j}],\phi[x_i\mapsto n_i,x\mapsto n_{a+1},X_j\mapsto X_j^{\ell_j}]$. We then obtain $d^\infty_{n_1,\ldots,n_a,\ell_1,\ldots,\ell_b}$ by applying an $\omega$ rule to these deductions.

  If $d(\langle\rangle)$ is Ind$^t_{\forall x\,\phi}$, let $\phi'=\phi[x_i\mapsto n_i,X_j\mapsto X_j^{\ell_j}]$. We will construct deductions $d^\infty_n$ with $\Gamma(d^\infty_n)\subseteq {\sim}\phi'(0),\exists x\,(\phi'(x)\wedge{\sim}\phi'(Sx)),\phi'(n)$ by induction on $n$, and then obtain $d^\infty_{n_1,\ldots,n_a,\ell_1,\ldots,\ell_b}$ by applying an $\omega$ rule to these deductions.  When $n=0$, $d^\infty_0$ is $d_{\phi'(0)}$.

  Suppose we have constructed $d^\infty_n$. Then we may take $d^\infty_{n+1}$ to be

  \AxiomC{$d^\infty_n$}
  \noLine
  \UnaryInfC{$\vdots$}
  \noLine
  \UnaryInfC{$ {\sim}\phi'(0), \exists x\,\phi'(x)\wedge{\sim}\phi'(Sx), \phi'(n)$}
  \AxiomC{$d_{\phi'(Sn)}$}
  \noLine
  \UnaryInfC{$\vdots$}
  \noLine
  \UnaryInfC{${\sim}\phi'(Sn),\phi'(Sn)$}
  \BinaryInfC{$ {\sim}\phi'(0), \exists x\,\phi'(x)\wedge{\sim}\phi'(Sx), \phi'(n)\wedge{\sim}\phi'(Sn), \phi'(Sn)$}
  \UnaryInfC{$ {\sim}\phi'(0), \exists x\,\phi'(x)\wedge{\sim}\phi'(Sx),\phi'(Sn)$}
  \DisplayProof

  If $d(\langle\rangle)$ is I$\forall^Y$ then we may choose a suitable $\ell$ and apply the inductive hypothesis.

  Finally, if $d(\langle\rangle)$ is I$\exists^\psi_{\exists^2X\,\phi}$, letting $\phi'=\phi[x_i\mapsto n_i,X_j\mapsto X_j^{\ell_j}]$, $\psi'=\psi[x_i\mapsto n_i,X_j\mapsto X_j^{\ell_j}]$, and $d'$ be the result of the inductive hypothesis, we have

  \AxiomC{$d'$}
  \noLine
  \UnaryInfC{$\vdots$}
  \noLine
  \UnaryInfC{$\phi[\psi]$}
  \AxiomC{$F^{{\sim}\phi,X^\ell\mapsto\psi}$}
  \RightLabel{$\Omega^\flat$}
  \UnaryInfC{$\exists^2X\,\phi, {\sim}\phi[\psi]$}
  \RightLabel{Cut}
  \BinaryInfC{$\exists^2X\,\phi$}
  \DisplayProof
\end{proof}

\subsection{Cut-Elimination}

We describe the cut-elimination operators. These definitions below are completely standard, expressed here in the language of locally defined functions.

We need one final theory.
\begin{definition}
  $PA^{\infty,N,L,-}_{2,<r}$ is the result of removing all Read rules from $PA^{\infty,N,L}_{2,<r}$.
\end{definition}

\begin{lemma}[$\bot$ Inversion]
  If $\eta$ is a false literal then there is a locally defined function $\mathrm{Inverse}^\bot_{\eta}$ on $PA^{\infty,N,L,-}_{2,<r}$ in $PA^{\infty,N,L}_{2,<r}$ with $\Gamma(\mathrm{Inverse}^\bot_{\eta})\subseteq (\eta,\langle\rangle,\eta)$.
\end{lemma}
\begin{proof}
  We define $\mathrm{Inverse}^\bot_{\eta}(\sigma)$ by induction on $\sigma$. We set $\mathrm{Inverse}^\bot_{\eta} (\langle\rangle)=Read_{PA^{\infty,N,L,-}_{2,<r},\eta,\epsilon,\eta}$.

  Suppose $\mathrm{Inverse}^\bot_{\eta} (\sigma)=Read_{PA^{\infty,N,L,-}_{2,<r},\eta,\epsilon,\eta}$. If $\mathcal{R}$ is anything other than Ax$_{\{\eta,{\sim}\eta\}}$ then $\mathrm{Inverse}^\bot_{\eta} (\sigma\mathcal{R})=\mathcal{R}$ and $\mathrm{Inverse}^\bot_{\eta} (\sigma\mathcal{R}\iota)=Read_{PA^{\infty,N,L,-}_{2,<r},\eta,\epsilon\iota,\eta}$.

  If $\mathcal{R}$ is Ax$_{\{\eta,{\sim}\eta\}}$ then $\mathrm{Inverse}^\bot_{\eta} (\sigma\mathcal{R})=True_{{\sim}\eta}$.
\end{proof}

\begin{lemma}[$\wedge$ Inversion]
For any $\phi_L,\phi_R$ and any $B\in\{L,R\}$ there is a locally defined function $\mathrm{Inverse}^\wedge_{\phi_L,\phi_R,B}$ on $PA^{\infty,N,L,-}_{2,<r}$ in $PA^{\infty,N,L}_{2,<r}$ with $\Gamma(\mathrm{Inverse}^\wedge_{\phi_L,\phi_R,B})\subseteq \phi_B;(\phi_L\wedge\phi_R,\epsilon_\bot,\phi_L\wedge\phi_R)$.
\end{lemma}
\begin{proof}
  As always, we define $\mathrm{Inverse}^\wedge_{\phi_L,\phi_R,B}$ by induction on $\sigma$. We set $\mathrm{Inverse}^\wedge_{\phi_L,\phi_R,B}(\langle\rangle)=Read_{PA^{\infty,N,L,-}_{2,<r},\phi_L\wedge\phi_R,\epsilon_\bot,\phi_L\wedge\phi_R}$.

  If $\mathrm{Inverse}^\wedge_{\phi_L,\phi_R,B}(\sigma)=Read_{PA^{\infty,N,L,-}_{2,<r},\phi_L\wedge\phi_R,\epsilon,\phi_L\wedge\phi_R}$ and $\mathcal{R}$ is anything other than I$\wedge_{\phi_L\wedge\phi_R}$ then $\mathrm{Inverse}^\wedge_{\phi_L,\phi_R,B}(\sigma\mathcal{R})=\mathcal{R}$ and $\mathrm{Inverse}^\wedge_{\phi_L,\phi_R,B}(\sigma\mathcal{R}\iota)=Read_{PA^{\infty,N,L,-}_{2,<r},\phi_L\wedge\phi_R,\epsilon\iota,\phi_L\wedge\phi_R}$.

  So suppose $\mathcal{R}$ is I$\wedge_{\phi_L\wedge\phi_R}$. Then we set $\mathrm{Inverse}^\wedge_{\phi_Bb}(\sigma\mathcal{R})=Read_{PA^{\infty,N,L,-}_{2,<r},\phi_L\wedge\phi_R,\epsilon B,\phi_L\wedge\phi_R}$.
\end{proof}

\begin{lemma}[$\forall$ Inversion]
  For any $\phi(x)$ and any $n$ there is a locally defined function $\mathrm{Inverse}^\forall_{\phi,x,n}$ on $PA^{\infty,N,L,-}_{2,<r}$ in $PA^{\infty,N,L}_{2,<r} $ with $\Gamma(\mathrm{Inverse}^\forall_{\phi,x,n})\subseteq \phi(n);(\forall x\,\phi,\epsilon_\bot,\forall x\,\phi)$.
\end{lemma}
\begin{proof}
  We set $\mathrm{Inverse}^\forall_{\phi,x,n}(\langle\rangle)=Read_{PA^{\infty,N,L,-}_{2,<r},\forall x\,\phi,\epsilon_\bot,\forall x\,\phi}$.

  If $\mathrm{Inverse}^\forall_{\phi,x,n}(\sigma)=Read_{PA^{\infty,N,L,-}_{2,<r},\forall x\,\phi,\epsilon,\forall x\,\phi}$ and $\mathcal{R}$ is anything other than $\omega_{\forall x\,\phi}$ then $\mathrm{Inverse}^\forall_{\phi,x,n}(\sigma\mathcal{R})=\mathcal{R}$ and $\mathrm{Inverse}^\forall_{\phi,x,n}(\sigma\mathcal{R}\iota)=Read_{PA^{\infty,N,L,-}_{2,<r},\forall x\,\phi,\epsilon\iota,\forall x\,\phi}$.

  So suppose $\mathcal{R}$ is $\omega_{\forall x\,\phi}$. Then we set $\mathrm{Inverse}^\forall_{\phi,x,n}(\sigma\mathcal{R})=Read_{PA^{\infty,N,L,-}_{2,<r},\forall x\,\phi,\epsilon n,\forall x\,\phi}$.
\end{proof}

\begin{lemma}[$\vee$ Elimination]
  For any $\phi_0,\phi_1$ there is a locally defined function $\mathrm{Elim}^\vee_{\phi_0,\phi_1}$ on $PA^{\infty,N,L,-}_{2,<r}$ in $PA^{\infty,N,L}_{2,<r}$ with $\Gamma(\mathrm{Elim}^\vee_{\phi_0,\phi_1})\subseteq (\phi_0\vee\phi_1,\epsilon_\bot,\phi_0\vee\phi_1), ({\sim}\phi_0\wedge{\sim}\phi_1,\epsilon_\bot, {\sim}\phi_0\wedge{\sim}\phi_1)$.
\end{lemma}
\begin{proof}
  We set $\mathrm{Elim}^\vee_{\phi_0,\phi_1}(\langle\rangle)=Read_{PA^{\infty,N,L,-}_{2,<r},\phi_0\vee\phi_1,\epsilon_\bot,\phi_0\vee\phi_1}$.

  If $\mathrm{Elim}^\vee_{\phi_0,\phi_1}(\sigma)=Read_{PA^{\infty,N,L,-}_{2,<r},\phi_0\vee\phi_1,\epsilon,\phi_0\vee\phi_1}$ and $\mathcal{R}$ is anything other than I$\vee_{\phi_0\vee\phi_1}$, $\mathrm{Elim}^\vee_{\phi_0,\phi_1}(\sigma\mathcal{R})=\mathcal{R}$ and $\mathrm{Elim}^\vee_{\phi_0,\phi_1}(\sigma\mathcal{R}\iota)=Read_{PA^{\infty,N,L,-}_{2,<r},\phi_0\vee\phi_1,\epsilon\iota,\phi_0\vee\phi_1}$.

  If $\mathcal{R}$ is I$\vee^B_{\phi_0\vee\phi_1}$ then $\mathrm{Elim}^\vee_{\phi_0,\phi_1}(\sigma\mathcal{R})=Cut_{\phi_B}$, $\mathrm{Elim}^\vee_{\phi_0,\phi_1}(\sigma\mathcal{R}0)=Read_{PA^{\infty,N,L,-}_{2,<r},\phi_0\vee\phi_1,\epsilon\top,\phi_0\vee\phi_1}$, and for all $\tau$, $\mathrm{Elim}^\vee_{\phi_0,\phi_1}(\sigma\mathcal{R}1\tau)=\mathrm{Inverse}^\wedge_{\phi_0,\phi_1,B}(\tau)$.
\end{proof}

\begin{lemma}[$\exists$ Elimination]
  For any $\phi(x)$ there is a locally defined function $\mathrm{Elim}^\forall_{\phi,x}$ on $PA^{\infty,N,L,-}_{2,<r}$ in $PA^{\infty,N,L}_{2,<r}$ with $\Gamma(\mathrm{Elim}^\forall_{\phi,x})\subseteq (\exists x\,\phi,\epsilon_\bot, \exists x\,\phi), (\forall x\,{\sim}\phi,\epsilon_\bot, \forall x\,{\sim}\phi)$.
\end{lemma}
\begin{proof}
  We set $\mathrm{Elim}^\forall_{\phi,x} (\langle\rangle)=Read_{PA^{\infty,N,L,-}_{2,<r},\exists x\,\phi,\epsilon_\bot, \exists x\,\phi}$.

  If $\mathrm{Elim}^\forall_{\phi,x} (\sigma)=Read_{PA^{\infty,N,L,-}_{2,<r},\exists x\,\phi,\epsilon, \exists x\,\phi}$ and $\mathcal{R}$ is anything other than I$\exists_{\exists x\,\phi}$, $\mathrm{Elim}^\forall_{\phi,x}(\sigma\mathcal{R})=\mathcal{R}$ and $\mathrm{Elim}^\forall_{\phi,x} (\sigma\mathcal{R}\iota)=Read_{PA^{\infty,N,L,-}_{2,<r},\exists x\,\phi,\epsilon\iota, \exists x\,\phi}$.

  If $\mathcal{R}$ is I$\exists^n_{\exists x\,\phi}$ then $\mathrm{Elim}^\forall_{\phi,x}(\sigma\mathcal{R})=Cut_{\phi_0}$, $\mathrm{Elim}^\forall_{\phi,x}(\sigma\mathcal{R}0)=Read_{PA^{\infty,N,L,-}_{2,<r},\exists x\,\phi,\epsilon\top, \exists x\,\phi}$, and, for all $\tau$, $\mathrm{Elim}^\forall_{\phi,x} (\sigma\mathcal{R}1\tau)=\mathrm{Inverse}^\forall_{{\sim}\phi,x,n}(\tau)$.
\end{proof}

The next two elimination steps don't have inversion lemmas, so we have do a two sided process where we work our way back to introduction rules on both sides. 
\begin{lemma}[$X^\ell$ Elimination]
  For any $X^\ell$, $m$ there is a locally defined function $\mathrm{Elim}^{X}_{X^\ell,m}$ on $PA^{\infty,N,L,-}_{2,<r}$ in $PA^{\infty,N,L}_{2,<r}$ with $\Gamma(\mathrm{Elim}^{X}_{X^\ell,m})\subseteq (X^\ell m,\epsilon_\bot,X^\ell m), (\neg X^\ell m,\epsilon_\bot,\neg X^\ell m)$.
\end{lemma}
\begin{proof}
  We arbitrarily choose to proceed down the $X^\ell m$ branch first, so we set $\mathrm{Elim}^X_{X^\ell,m}(\langle\rangle)=Read_{PA^{\infty,N,L,-}_{2,<r},X^\ell m,\epsilon_\bot,X^\ell m}$.

  Suppose $\mathrm{Elim}^X_{X^\ell,m}(\sigma)=Read_{PA^{\infty,N,L,-}_{2,<r},X^\ell m,\epsilon,X^\ell m}$. If $\mathcal{R}$ is anything other than $Ax_{\{X^\ell m,\neg X^\ell m\}}$ then $\mathrm{Elim}^X_{X^\ell,m}(\sigma\mathcal{R})=\mathcal{R}$ and $\mathrm{Elim}^X_{X^\ell,m}(\sigma\mathcal{R}\iota) =Read_{PA^{\infty,N,L,-}_{2,<r},X^\ell m,\epsilon\iota,X^\ell m}$.

  If $\mathcal{R}$ is $Ax_{\{X^\ell m,\neg X^\ell m\}}$ then $\mathrm{Elim}^X_{X^\ell,m}(\sigma\mathcal{R})=Read_{PA^{\infty,N,L,-}_{2,<r},\neg X^\ell m,\epsilon_\bot,\neg X^\ell m}$.

  If $\mathrm{Elim}^X_{X^\ell,m}(\sigma)=Read_{PA^{\infty,N,L,-}_{2,<r},X^\ell m,\epsilon,X^\ell m}$ then, for any $\mathcal{R}$, $\mathrm{Elim}^X_{X^\ell,m}(\sigma\mathcal{R})=\mathcal{R}$ and $\mathrm{Elim}^X_{X^\ell,m}(\sigma\mathcal{R}\iota) =Read_{PA^{\infty,N,L,-}_{2,<r},\neg X^\ell m,\epsilon\iota,\neg X^\ell m}$.
\end{proof}

\begin{lemma}[$Q^2$ Elimination]
  For any $\phi$ there is a locally defined function $\mathrm{Elim}^{Q^2}_{\phi,X}$ on $PA^{\infty,N,L,-}_{2,<r}$ in $PA^{\infty,N,L}_{2,<r}$ with $\Gamma(\mathrm{Elim}^{Q^2}_{\phi,X})\subseteq (\forall^2X\,\phi,\epsilon_\bot,\forall^2X\,\phi), (\exists^2X\,{\sim}\phi,\epsilon_\bot,\exists^2X\, {\sim}\phi)$.
\end{lemma}
\begin{proof}
  Arbitrarily, we choose to pass down the $\forall^2$ branch first, so we set $\mathrm{Elim}^{Q^2}_{\phi,X}(\langle\rangle)=Read_{PA^{\infty,N,L,-}_{2,<r},\forall^2X\,\phi,\epsilon_\bot,\forall^2X\,\phi}$. We keep track of additional values $\epsilon(\sigma)$ and $\epsilon'(\sigma)$; initially $\epsilon(\langle\rangle)=\epsilon'(\langle\rangle)=\langle\rangle$.

  Suppose $\mathrm{Elim}^{Q^2}_{\phi,X}(\sigma)=Read_{PA^{\infty,N,L,-}_{2,<r},\forall^2X\,\phi,\epsilon,\forall^2X\,\phi}$. If $\mathcal{R}$ is anything other than I$\forall^{\{Y^\ell\}}_{\forall^2X\,\phi}$ then $\mathrm{Elim}^{Q^2}_{\phi,X}(\sigma\mathcal{R})=\mathcal{R}$ and $\mathrm{Elim}^{Q^2}_{\phi,X}(\sigma\mathcal{R}\iota)=Read_{PA^{\infty,N,L,-}_{2,<r},\forall^2X\,\phi,\epsilon\iota,\forall^2X\,\phi}$.  Set $\epsilon(\sigma\mathcal{R}\iota)=\epsilon(\sigma)\iota$ and $\epsilon'(\sigma\mathcal{R}\iota)=\epsilon'(\sigma)$.

  If $\mathcal{R}$ is I$\forall^{Y^\ell}_{\forall^2X\,\phi}$ then $\mathrm{Elim}^{Q^2}_{\phi,X}(\sigma\mathcal{R})=Read_{PA^{\infty,N,L,-}_{2,<r},\exists^2X\, {\sim}\phi,\epsilon'(\sigma),\exists^2X\, {\sim}\phi}$. We set $\epsilon(\sigma\mathcal{R})=\epsilon(\sigma)$ and $\epsilon'(\sigma\mathcal{R})=\epsilon'(\sigma)$.

  If $\mathrm{Elim}^{Q^2}_{\phi,X}(\sigma)=Read_{PA^{\infty,N,L,-}_{2,<r},\exists^2X\, {\sim}\phi,\epsilon',\exists^2X\, {\sim}\phi}$ and $\mathcal{R}$ is anything other than $\Omega^\flat_{\mathfrak{T},Y^\ell,\exists^2X\, {\sim}\phi}$ then $\mathrm{Elim}^{Q^2}_{\phi,X}(\sigma\mathcal{R})=\mathcal{R}$ and $\mathrm{Elim}^{Q^2}_{\phi,X}(\sigma\mathcal{R}\iota)=Read_{PA^{\infty,N,L,-}_{2,<r},\exists^2X\, {\sim}\phi,\epsilon'\iota,\exists^2X\, {\sim}\phi}$. In this case we set $\epsilon(\sigma\mathcal{R}\iota)=\epsilon(\sigma)$ and $\epsilon'(\sigma\mathcal{R}\iota)=\epsilon'(\iota)$.

  If $\mathcal{R}$ is $\Omega^\flat_{\mathfrak{T},Y^\ell,\exists^2X\, {\sim}\phi}$ then we set $\mathrm{Elim}^{Q^2}_{\phi,X}(\sigma\mathcal{R})=Cut\Omega^\flat_{\mathfrak{T},Z^\ell(\sigma),Y^\ell,\exists^2X\, {\sim}\phi}$. We set $\mathrm{Elim}^{Q^2}_{\phi,X}(\sigma\mathcal{R}\top)=Read_{PA^{\infty,N,L,-}_{2,<r},\forall^2X\,\phi,\epsilon(\sigma)\top,\forall^2X\,\phi}$ and $\mathrm{Elim}^{Q^2}_{\phi,X}(\sigma\mathcal{R}\bot)=Read_{PA^{\infty,N,L,-}_{2,<r},\exists^2X\, {\sim}\phi,\epsilon'(\sigma)\bot,\exists^2X\, {\sim}\phi}$. We set $\epsilon(\sigma\mathcal{R}\top)=\epsilon(\sigma)\ell$, $\epsilon'(\sigma\mathcal{R}\top)=\epsilon'(\sigma)$, $\epsilon(\sigma\mathcal{R}\bot)=\epsilon(\sigma)$, and $\epsilon'(\sigma\mathcal{R}\bot)=\epsilon'(\sigma)\bot$.
\end{proof}

\begin{theorem}
  There is a locally defined function $\mathrm{Reduce}_r$ on $PA^{\infty,N,L,-}_{2,<r+1}$ in $PA^{\infty,N,L}_{2,<r}$ with $\Gamma(\mathrm{Reduce}_r)\subseteq (\emptyset,\epsilon_\bot,\emptyset)$.
\end{theorem}
\begin{proof}
  The natural way to describe this function is coinductively. For any $\epsilon$, we can describe the operation starting at position $\epsilon$, and we do this simultaneously for all $\epsilon$, using the operations defined above.

  We define $\mathrm{Reduce}_{r,\epsilon}(\langle\rangle)$ to be Read$_{PA^{\infty,N,L,-}_{2,<r+1},\emptyset,\epsilon,\emptyset}$. Suppose we have defined $\mathrm{Reduce}_{r,\epsilon}(\sigma)$ to be Read$_{PA^{\infty,N,L,-}_{2,<r+1},\emptyset,\epsilon\epsilon',\emptyset}$ and consider some $\mathcal{R}$.

  If $\mathcal{R}$ is anything other than Cut$_\phi$ with $rk(\phi)=r$, we set $\mathrm{Reduce}_{r,\epsilon}(\sigma\mathcal{R})=\mathcal{R}$ and $\mathrm{Reduce}_{r,\epsilon}(\sigma\mathcal{R}\iota)=Read_{PA^{\infty,N,L,-}_{2,<r},\emptyset,\epsilon\epsilon'\iota,\emptyset}$.

  So suppose $\mathcal{R}$ is Cut$_\phi$ where $rk(\phi)=r$. We consider cases on $\phi$. If $\phi$ is $\psi_0\wedge\psi_1$, we take $\mathrm{Reduce}_{r,\epsilon}(\sigma\mathcal{R}\tau)$ to be $\mathrm{Elim}^\vee_{{\sim}\psi_0,{\sim}\psi_1}(\mathrm{Reduce}_{r,\epsilon 1},\mathrm{Reduce}_{r,\epsilon 0})(\tau)$. Similarly, if $\phi$ is $\psi_0\vee\psi_1$, we take $\mathrm{Reduce}_{r,\epsilon}(\sigma\mathcal{R}\tau)$ to be $\mathrm{Elim}^\vee_{\psi_0,\psi_1}(\mathrm{Reduce}_{r,\epsilon 0},\mathrm{Reduce}_{r,\epsilon 1})(\tau)$.

  If $\phi$ is $\forall x\,\psi$, we take $\mathrm{Reduce}_{r,\epsilon}(\sigma\mathcal{R}\tau)$ to be $\mathrm{Elim}^\exists_{{\sim}\psi,x}(\mathrm{Reduce}_{r,\epsilon 1},\mathrm{Reduce}_{r,\epsilon 0})(\tau)$. Similarly, if $\phi$ is $\exists x\,\psi$, we take $\mathrm{Reduce}_{r,\epsilon}(\sigma\mathcal{R}\tau)$ to be $\mathrm{Elim}^\exists_{\psi,x}(\mathrm{Reduce}_{r,\epsilon 0},\mathrm{Reduce}_{r,\epsilon 1})(\tau)$.

  If $\phi$ is a literal $\dot{R}t_1\cdots t_m$ then, since it is closed, it is true or false, so we take $\mathrm{Reduce}_{r,\epsilon}(\sigma\mathcal{R}\tau)$ to be $\mathrm{Inverse}^\bot_{{\sim}\phi}(\mathrm{Reduce}_{r,\epsilon 1})(\tau)$ or $\mathrm{Inverse}^\bot_{\phi}(\mathrm{Reduce}_{r,\epsilon 0})(\tau)$ respectively.

  If $\phi$ is a literal $X^\ell m$ then we take $\mathrm{Reduce}_{r,\epsilon}(\sigma\mathcal{R}\tau)$ to be $\mathrm{Elim}^X_{X^\ell,n}(\mathrm{Reduce}_{r,\epsilon 0}, \mathrm{Reduce}_{r,\epsilon 1})(\tau)$, and if $\phi$ is $\neg X^\ell m$ then we take $\mathrm{Reduce}_{r,\epsilon}(\sigma\mathcal{R}\tau)$ to be $\mathrm{Elim}^X_{X^\ell,n}(\mathrm{Reduce}_{r,\epsilon 1}, \mathrm{Reduce}_{r,\epsilon 0})(\tau)$.

  If $\phi$ is $\forall^2X\,\psi$ then we take $\mathrm{Reduce}_{r,\epsilon}(\sigma\mathcal{R}\tau)$ to be $\mathrm{Elim}^{Q^2}_{\phi,X}(\mathrm{Reduce}_{r,\epsilon 1}, \mathrm{Reduce}_{r,\epsilon 0})(\tau)$, and if $\phi$ is $\exists^2X\,\psi$ then we take $\mathrm{Reduce}_{r,\epsilon}(\sigma\mathcal{R}\tau)$ to be $\mathrm{Elim}^{Q^2}_{\phi,X}(\mathrm{Reduce}_{r,\epsilon 0}, \mathrm{Reduce}_{r,\epsilon 1})(\tau)$.

  We take $\mathrm{Reduce}_r$ to be $\mathrm{Reduce}_{r,\langle\rangle}$.
\end{proof}

\begin{cor}
  If $d$ is a deduction in $PA^{\infty,N,L}_{2,<r}$ then there is a deduction $d'$ in $PA^{\infty,n}_{2,<0}$ with $\Gamma(d')\subseteq\Gamma(d)$.
\end{cor}
\begin{proof}
  By Lemma \ref{operator_extension}, we may extend each $\mathrm{Reduce}_i$ to a function on $PA^{\infty,N,L}_{2,<i+1}$. We therefore take $d'=\overline{\mathrm{Reduce}_0^\uparrow}(\cdots \overline{\mathrm{Reduce}_{r-1}^\uparrow}(d))$.
\end{proof}

\subsection{Collapsing}

The remaining step in a proof of cut-elimination is collapsing: showing that we can remove Cut$\Omega^\flat$ inference rules in our deduction.

\begin{theorem}
  Let $d$ be a proof-tree in $PA^{\infty,N,L,m,\ell}_{2,<0}$ with $m+\ell>0$ such that
\begin{itemize}
\item   whenever $\exists^2 X\,\phi$ appears as a subformula of a formula in $\Gamma(d)$, $comp(\exists^2X\,\phi)<(m,\ell)$,
\item if $t\in\Gamma(d)_{\mathsf{t}}$ then $comp(\exists^2X\,\Theta_t)<(m,\ell)$.
\end{itemize}
Let $(m',\ell')$ be the predecessor of $(m,\ell)$. Then there is a proof-tree $D_{m,\ell}(d)$ in $PA^{\infty,N,L,m',\ell'}_{2,<0}$ with $\Gamma(D_{m,\ell}(d))\subseteq\Gamma(d)$.
\end{theorem}
$D_{m,\ell}$ is, itself, essentially a locally defined function modulo some technicalities about consecutive Read rules, but we do not need to represent it as one, so skip the technicalities to represent it as one.

$d$ cannot contain $\Omega^\flat$ inferences of complexity $(m,\ell)$, so the issue is collapsing Cut$\Omega^\flat$ inferences. When $d(\langle\rangle)$ is a Cut$\Omega^\flat$ inference of complexity $(m,\ell)$, we would like to view $d_{\langle\bot\rangle}$ as a locally defined function on domain $PA^{\infty,N,L,m,\ell,-}_{2,<0}$. We should be able to collapse $d_{\langle\top\rangle}$, giving us $D_{m,\ell}(d_{\langle\top\rangle})$ in $PA^{\infty,N,L,m',\ell'}_{2,<0}$. We can lift our function $d_{\langle\bot\rangle}$ to this domain, so we should end up with $D_{m,\ell}(\overline{d_{\langle\bot\rangle}^{\uparrow}}(D_{m,\ell}(d_{\langle\top\rangle})))$. Composing, this is the same as $\overline{D_{m,\ell}(d_{\langle\bot\rangle})^\uparrow}(D_{m,\ell}(d_{\langle\top\rangle}))$. The steps below represent precisely this process.

The restrictions on $\Gamma(d)$ do not hold inductively---$d$ could contain a Cut$\Omega^\flat$ with complexity $(m',\ell')<(m,\ell)$, and then an $\Omega^\flat$ above that introducing a formula with complexity $(m'',\ell'')$ where $\ell'$ is high even though $m''$ is low. However the restriction that the conclusion of this $\Omega^\flat$ rule is a subformula of our low complexity Cut$\Omega^\flat$ rule ensures that no $X^\ell$ appears free, which is all we really need.
\begin{proof}
  Again, we need to define ``collapsing above $\epsilon$'' as well. We will define $D_{m,\ell,\epsilon}(d)(\sigma)$ by induction on $\sigma$, maintaining, for each $t$ in $\Gamma^{\leftarrow}(\sigma)$ with $comp(\exists^2X\,\Theta_t)=(m,\ell)$, a $\pi_{t,\epsilon}(\sigma)$ in $\dom(d)$.

  We initially set $h(\langle\rangle)=\epsilon$. Suppose we have defined $h(\sigma)$; we now wish to define $D_{m,\ell}(d)(\sigma)$. If $d(h(\sigma))$ is anything other than a Cut$\Omega^\flat_{Z^\ell,Y^\ell,\exists^2X\,\phi}$ with $comp(\exists^2X\,\phi)=(m,\ell)$ or a Read$_{\mathfrak{T},t}$ with $comp(\exists^2X\,\Theta_t)=(m,\ell)$ then we set $D_{m,\ell,\epsilon}(d)(\sigma)=d(h(\sigma))$ and, for each $\iota$, $h(\sigma\iota)=h(\sigma)\iota$ and for each $t$, $\pi_t(\sigma\iota)=\pi_t(\sigma)$.

  Suppose $d(h(\sigma))$ is a Cut$\Omega^\flat_{Z^\ell,Y^\ell,\exists^2X\,\phi}$ with $comp(\exists^2X\,\phi)=(m,\ell)$. Then we set $D_{m,\ell}(d)(\sigma)=Rep$, $h(\sigma\top)=h(\sigma)\bot$, and $\pi_{{\sim}\phi,\langle\rangle,{\sim}\phi}(\sigma\top)=\sigma\bot$.

  Suppose $d(h(\sigma))$ is some Read$_{\mathfrak{T},t}$ with $comp(\exists^2X\,\Theta_t)=(m,\ell)$.   Let $\mathcal{R}$ be $D_{m,\ell,\pi_{t,\epsilon}(\sigma)}(\epsilon_t)$. If $\mathcal{R}$ belongs to $PA^{\infty,N,L,m,\ell,-}_{2,<0}$ then we extend $h(\sigma)=h(\sigma)\mathcal{R}$ and continue. Otherwise, we set $D_{m,\ell,\epsilon}(d)(\sigma)$ to be the rule given by Definition \ref{def:lifting} (that is, $\mathcal{R}$ if $\mathcal{R}$ is not a Read rule, the modification of $\mathcal{R}$ if $\mathcal{R}$ is a Read rule) and, for each $\iota$, $h(\sigma\iota)=h(\sigma)Rep$.

  We can then take $D_{m,\ell}(d)$ to be $D_{m,\ell,\langle\rangle}(d)$.
\end{proof}

Putting these results together gives the following.
\begin{theorem}
  If there is a deduction of an arithmetic sequent $\Gamma$ in $PA_2$ then, for some $N,L$, there is a proof-tree $d$ in $PA_{2,<0}^{\infty,N,L,0,0}$ with $\Gamma(d)\subseteq\Gamma$.
\end{theorem}
This is rather unsatisfying on its own since $d$ might be ill-founded. The next section shows that we can improve this to a well-founded deduction.

\section{Ordinals, Qualitatively}

To complete a proof of cut-elimination, we will need to restrict to our analog of well-founded deductions---deductions with ordinal bounds, in the sense we now introduce. In many cases, these bounds will be functions on ordinals, not simply ordinals. In this section we carry out the ``qualitative'' part of the argument, that is, the part that can be carried out without a formal ordinal notation.

\subsection{Functions on Ordinals}

We consider certain \emph{sorts}:
\begin{itemize}
\item $Ord$ is the sort of ordinals,
\item if $\theta,\theta'$ are sorts then $\theta\rightarrow\theta'$ is the sort of functions from sort $\theta$ to sort $\theta'$,
\item if $I$ is a set and, for each $i\in I$, $\theta_i$ is a sort then $\prod_{i\in I}\theta_i$ is a sort.
\end{itemize}

\begin{definition}
  When $f,g$ have the same sort, we define $f<g$ by induction on the sort:
  \begin{itemize}
  \item if the sort of $f,g$ is $Ord$ then $f<g$ is the usual comparison of ordinals,
  \item if $f,g$ have sort $\theta\rightarrow\theta'$ then $f<g$ if there is some $a$ of sort $\theta$ so that whenever $a<b$, $f(a)<g(a)$,
  \item if $f,g$ have sort $\prod_{i\in I}\theta_i$ then $f<g$ if, for each $i\in I$, $f_i<g_i$.
  \end{itemize}
\end{definition}

\begin{lemma}
In every sort, pairs have common upper bounds and increasing chains have upper bounds.
\end{lemma}
\begin{proof}
  Both proofs are by induction on sorts. The casse where the sort is Ord is immediate, as are the cases of a product sort.

  Consider a sort $\theta\rightarrow\theta'$. Given $f,g$, we can use the inductive hypothesis to define $h(a)$ to be an upper bound of $f(a),g(a)$ for each $a$ of sort $\theta$.

  Suppose we have $f_0<f_1<\cdots$ in this sort. We choose an increasing sequence so $a_i$ is such that $b>a_i$ implies $f_i(b)<f_{i+1}(b)$. (At successors, we can use the common upper bound with $a_i$ to make sure $a_i<a_{i+1}$, and at limits we can use the inductive hypothesis followed by the common upper bound.) Finally, we use the inductive hypothesis, giving us an $a$ so that whenever $b>a$, $f_i(b)<f_{i+1}(b)$ for all $i$. Then we define $f$ arbitrarily except for $b>a$, and for $b>a$ we choose some $f(b)>f_i(b)$ for all $i$ using the inductive hypothesis.
\end{proof}

\begin{lemma}
  Every sort is well-founded under $<$.
\end{lemma}
\begin{proof}
  By induction on sorts. The case where the sort is Ord is immediate, as is the case of a product sort. For the function case, if we have $f_0>f_1>\cdots$, we can, as in the previous lemma, choose an $a$ so that $f_i(b)>f_{i+1}(b)$ for all $b>a$ and all $i$. But this gives us a decreasing sequence contradicting the inductive hypothesis.
\end{proof}

\subsection{Ordinal Bounds}

\begin{definition}
  Let $\sigma$ be a $PA^{\infty,N,L,m,\ell}_{2,<r}$-sequence and let $T$ be a set of tags. We define $\mathcal{T}_T(\sigma)$, the \emph{open tags at $\sigma$}, inductively by:
  \begin{itemize}
  \item $\mathcal{T}_T(\langle\rangle)=\{\top_{m',\ell'}\}_{m'<m,\ell'\leq L}\cup \{t\in T\mid \mathsf{r}(t)=\{\exists^2X\,\phi\}\text{ where }dp(\phi,\{X\})<m\}$,
  \item if $\mathcal{R}(\iota)$ is an $\Omega^\flat_{X^\ell,\exists^2X\,\phi}$ rule with $dp(\phi,{X})<m$ then $\mathcal{T}_T(\sigma\iota)=\mathcal{T}_T(\sigma)\cup\{({\sim}\phi(X^\ell),\langle\rangle,{\sim}\phi(X^\ell))\}$,
  \item if $\mathcal{R}(\iota)$ is an Cut$\Omega^\flat_{X^\ell,\exists^2X\,\phi}$ rule with $dp(\phi,{X})<m$ and $\iota=\bot$ then $\mathcal{T}_T(\sigma\iota)=\mathcal{T}_T(\sigma)\cup\{({\sim}\phi(X^\ell),\langle\rangle,{\sim}\phi(X^\ell))\}$,
  \item if $\mathcal{R}(\iota)$ is a Read$_t$ rule with $t\in\mathcal{T}_T(\sigma)$ then $\mathcal{T}_T(\sigma\iota)=\mathcal{T}_T(\sigma)\cup\{\Delta_\iota(\mathcal{R}_\iota)_{\mathsf{t}}\}$,
  \item otherwise $\mathcal{T}_T(\sigma\iota)=\mathcal{T}_T(\sigma)$.
  \end{itemize}

\end{definition}
The important fact is that we include each tag for which $PA^{\infty,N,L,m,\ell}_{2,<r}$ contains Read rules which branch over the theory. Importantly, we do \emph{not} include Root tags introduced at an $\Omega$ or Cut$\Omega$ for which the corresponding Read rules are not present. We also include a ``placeholder'' tag $\top_{m',\ell'}$ for each theory we might encounter tags for.

The extra tags $T$ account for the possibility that $\sigma$ belongs to a locally defined function; $T$ will typically be $\Gamma(d)_{\mathsf{t}}$.

\begin{definition}
  When $\sigma\sqsubseteq\tau$, there is a function $\pi:\mathcal{T}_T(\tau)\rightarrow\mathcal{T}_T(\sigma)$; for $t$ in $\mathcal{T}(\tau)$, $\pi(t)$ is the $t'\in\mathcal{T}_T(\sigma)$ with $\mathsf{r}(t')=\mathsf{r}(t)$ and $\epsilon_{t'}$ maximal so $\epsilon_{t'}\sqsubseteq\epsilon_t$ if there is one, and $\top_{comp(\mathsf{r}(t))}$ otherwise.
\end{definition}

\begin{definition}  
  When $T$ is a set of root tags and $\sigma$ is a $PA^{\infty,N,L,m,\ell}_{2,<r}$-sequence, we define sorts $\mathcal{O}^{\mathrm{prem}}_T(\sigma)$ and $\mathcal{O}_T(\sigma)$ by setting $\mathcal{O}^{\mathrm{prem}}_T(\sigma)$ to be $\prod_{t\in \mathcal{T}_T(\sigma)}\mathcal{O}_T(\epsilon(t))$ and $\mathcal{O}_T(\sigma)$ to be $\mathcal{O}_T^{\mathrm{prem}}(\sigma)\rightarrow Ord$.
  
  Given $\sigma\sqsubset\tau$ and $f_\sigma,f_\tau$ of sorts $\mathcal{O}^{\mathrm{prem}}_T(\sigma)$, $\mathcal{O}^{\mathrm{prem}}_T(\tau)$ respectively, we say $f_\tau< f_\sigma$ if, for each $t\in\mathcal{T}(\tau)$:
  \begin{itemize}
  \item if $\pi(t)=t$ then $(f_\tau)_t\leq (f_\sigma)_t$,
  \item if $\pi(t)\neq t$ then $(f_\tau)_t<(f_\sigma)_{\pi(t)}$.
  \end{itemize}

  When $d$ is a deduction, an \emph{ordinal bound on $d$} is an assignment, to each $\sigma\in\dom(d)$ so that $d(\sigma)$ is not a Read rule, of an $o_\sigma^d$ of sort $\mathcal{O}_{\Gamma(d)_{\mathsf{t}}}(\sigma)$ such that whenever $\sigma\sqsubset\tau$ and $f_\tau<f_\sigma$, $o^d_\tau(f_\tau)<o^d_\sigma(f_\sigma)$.
\end{definition}
When $m=0$, $\mathcal{T}_T(\sigma)=\emptyset$, so $\mathcal{O}^{\mathrm{prem}}_T(\sigma)$ is always an empty product---that is, a single point---and so $\mathcal{O}_T(\sigma)$ is an ordinal. That is, an ordinal bound on a deduction in $PA^{\infty,N,L,0,\ell}_{2,<r}$ is precisely a decreasing assignment of ordinals as usual.

When $m=1$, $\mathcal{T}_T(\sigma)$ is a set of tags for theories where $m=0$, so $\mathcal{O}^{\mathrm{prem}}_T(\sigma)$ is an indexed collection of ordinals. An ordinal bound therefore assigns, to each $\sigma$, a function from (possibly several) ordinals to an ordinal. As we will see when analyzing the cut-elimination operators below, these are familiar ordinal bounds on functions---for instance, when $d$ is the $\mathrm{Reduce}$ operator, we will be able to assign the function $\lambda\beta. \omega^\beta$ to each $\sigma$.

The decision to not assign bounds to Read rules reflects the fact that Read rules are ``ephemeral''---they will disappear as we evaluate functions. (This is related to why we need to prohibit consecutive Reads, because we need to avoid infinite branches in which no bounds get assigned.)

We adopt the convention that we write elements of $\mathcal{O}^{\mathrm{prem}}_T(\sigma)\rightarrow Ord$ as ordinal terms using variables indexed by $\mathcal{T}_T(\sigma)$. The variable indexed by $\top_{m,\ell}$ is written $\Omega_{m,\ell}$. For instance, we will write $\Omega_{0,1}^4+\Omega_{0,0}^2\cdot 2$ as an abbreviation for $\lambda\Omega_{1,1}\lambda\Omega_{0,0}. \Omega_{0,1}^4+\Omega_{0,0}^2\cdot 2$.

This notation means the sort may need to be inferred from context---we could also write $\Omega_{0,1}^4+\Omega_{0,0}^2\cdot 2$ for $\lambda\Omega_{0,1}\lambda\Omega_{1,0}\lambda\Omega_{0,0}. \Omega_{0,1}^4+\Omega_{0,0}^2\cdot 2$.

When $|T|=1$, we will write $\beta$ for the unique variable indexed by $T$; in the case $|T|>1$, we will attach subscripts to $\beta$ to indicate which variable we mean.

Finally, when $t$ is a term of sort $\theta\rightarrow Ord$, and $f$ is a term of type $Ord\rightarrow Ord$, we write $f(t)$ for $\lambda s. f(ts)$. For instance, we write $t+2$ for $\lambda s. ts+2$.

The following comparison observation is useful. Suppose $o$ is an ordinal bound on $d$ and $d(\langle\rangle)$ is Cut$\Omega^\flat$. We can view $o_{\langle\bot\rangle}$ as a bound of sort $\mathcal{O}_T(d_{\langle\top\rangle})\rightarrow\mathcal{O}_T(d)$, and observe that, for any $\alpha$, $o_{\langle\bot\rangle}(\alpha)<o_{\langle\rangle}$---the tag corresponding to the Cut$\Omega^\flat$ rule is the only tag present in $\mathcal{T}_T(\langle\bot\rangle)\setminus\mathcal{T}_T(\langle\rangle)$, so if we assign the same values to all other tags then, as soon as we choose $\Omega_{m,\ell}>\alpha$, the comparison requirement on an ordinal bound ensures $o_{\langle\bot\rangle}(\alpha,\{\beta_i\})<o_{\langle\rangle}(\{\beta_i\})$.

\subsection{Assigning Ordinal Bounds}

We next show that all the deductions and operations we have used so far can be assigned ordinal bounds.

\begin{lemma}
The locally defined function Id$^\phi$ has an ordinal bound given by $o_\sigma=\beta$.
\end{lemma}
Note that, according to our notation above, this is assigning the identity function $\lambda\beta. \beta$ (more precisely, we are assigning $\lambda\{\Omega_{m,\ell}\}\lambda\beta. \beta$).
\begin{proof}
We only need to check the comparison condition: suppose $\sigma\sqsubset\tau$ are in $\dom(\mathrm{Id}^\phi)$ and neither is a Read rule. Let $\{f^\sigma_t\}_{t\in\mathcal{T}_{(\phi,\epsilon_\bot,\phi)}(\sigma)}$ and $\{f^\tau_t\}_{t\in\mathcal{T}_{(\phi,\epsilon_\bot,\phi)}(\tau)}$ be given. By construction, there is a Read rule with root $\{\phi\}$ between $\sigma$ and $\tau$, so in particular $f^\tau_{\phi}<f^\sigma_{\phi}$, as needed.
\end{proof}

\begin{lemma}
  For any closed formula $\phi$ there are $N,L$ large enough so that the proof-tree $d_\phi$ has an ordinal bound $o^{d_\phi}$ with $o^{d_\phi}_\sigma=\max\{\Omega_{m,\ell}\}+k$.
\end{lemma}
We could restrict the maximum to those $m,\ell$ so that $\phi$ contains a subformula of the form $\exists^2 X\,\psi$ with $comp(\exists^2X\,\psi)=(m,\ell)$ (indeed, we could restrict to \emph{maximal} subformulas of this kind).

\begin{lemma}
  The locally defined function $F^{\phi,X^\ell\mapsto\psi}$ has an ordinal bound given by $o_\sigma=\beta(\{\Omega_{m',\ell'}\}_{m'<m})$ where $\ell'=lvl(\psi)$.
\end{lemma}

\begin{lemma}\label{thm:embedding_bounds}
  For any deduction $d$ in $PA_2$, the deduction $d^\infty_{n_1,\ldots,n_a,\ell_1,\ldots,\ell_b}$ has an ordinal bound with $o^{d^\infty}_{\langle\rangle}$ bounded by an expression of the form $\sum_i t_i$ where each $t_i$ is one of:
  \begin{itemize}
  \item a finite number,
  \item $\omega$,
  \item an expression $\Omega_{m,\ell}(\{\Omega_{m',\ell'}\}_{m'<m})$ for some $\ell,\ell'$.
  \end{itemize}
\end{lemma}

The bounds on the cut-elimination operators are standard.
\begin{theorem}\label{thm:cut_elim_bounds}
  \begin{enumerate}
  \item $\mathrm{Inverse}^\bot_\eta(\sigma)$ is bounded by $o_\sigma=\beta$,
  \item $\mathrm{Inverse}^\wedge_{\phi_L,\phi_R,B}$ is bounded by $o_\sigma=\beta$,
  \item $\mathrm{Inverse}^\forall_{\phi,x,n}(\sigma)$ is bounded by $o_\sigma=\beta$,
  \item $\mathrm{Elim}^\vee_{\phi_0,\phi_1}(\sigma)$ is bounded by $o_\sigma=\beta_\wedge+\beta_\vee$,
  \item $\mathrm{Elim}^\forall_{\phi, x}(\sigma)$ is bounded by $o_\sigma=\beta_\forall+\beta_\exists$,
  \item $\mathrm{Elim}^X_{X^\ell,m}(\sigma)$ is bounded by $o_\sigma=\beta_X\#\beta_{\neg X}$,
  \item $\mathrm{Elim}^{Q^2}_{\phi,X}(\sigma)$ is bounded by $o_\sigma=\beta_{\forall^2}\#\beta_{\exists^2}$,
  \item $\mathrm{Reduce}_r(\sigma)$ is bounded by $o_\sigma=\omega^\beta$.
  \end{enumerate}
\end{theorem}

\subsection{Calculating Ordinal Bounds}

Of course, our definitions have been set up to ensure that when we interpret a proof as a function, its ordinal bounds can be calculated using the same functions.

\begin{theorem}
  Let $F$ be a locally defined function on $PA^{\infty,N,L,m,\ell}_{2,<r}$ with ordinal bound $o^F$ and $d$ a deduction in $PA^{\infty,N,L,m,\ell}_{2,<r}$ with ordinal bound $o^d$. Then $\bar F(d)$ has an ordinal bound given by $o^{\bar F(d)}_\sigma=o^F_{h(\sigma)}(\{o^d_{\epsilon(t)}\}_{t\mid \mathsf{r}(t)=t_F})$.
\end{theorem}
\begin{proof}
We need only check the comparison rule. Suppose $\sigma\sqsubset\tau$ and $f_\tau<f_\sigma$. We can lift these to tuples $f'_{h(\tau)},f'_{h(\sigma)}$ by $(f'_{h(\tau)})_t=(f_\tau)_t$ for $t\in \mathcal{T}_T(\tau)$ and $(f'_{h(\tau)})_{t}=o^d_{\epsilon_t}$ when $\mathsf{r}(t)=t_F$, and similarly for $\sigma$. Then $f'_{h(\tau)}<f'_{h(\sigma)}$, so $o^{\bar F(d)}_{\tau}(f_\tau)=o^F_{h(\tau)}(f_{h(\tau)},o^d_{\epsilon(\tau)})=o^F_{h(\tau)}(f'_{h(\tau)})<o^F_{h(\sigma)}(f'_{h(\sigma)})=o^{\bar F(d)}_\sigma(f_\sigma)$.
\end{proof}

We have an additional issue: some of our functions have the ``wrong'' domains---we defined the function on one domain, but need to use lifting to get it to the domain it needs to operate on. We need to show that we can lift ordinal bounds as well.
\begin{theorem}\label{thm:bound_lifting}
  Let $F$ be a locally defined function with ordinal bound $o^F$ and let $F^\uparrow$ be the lift. Then the the bound $o^{F^\uparrow}_\sigma=o^F_{h(\sigma)}$ is an ordinal bound on $F^\uparrow$.
\end{theorem}

\subsection{Collapsing}

\begin{theorem}\label{ref:collapsing_bounds}
  Let $d$ be a proof-tree in $PA^{\infty,N,L,m,\ell}_{2,<0}$ with $m+\ell>0$ such that
\begin{itemize}
\item   whenever $\exists^2 X\,\phi$ appears as a subformula of a formula in $\Gamma(d)$, $comp(\exists^2X\,\phi)<(m,\ell)$,
\item if $t\in\Gamma(d)_{\mathsf{t}}$ then $comp(\exists^2X\,\Theta_t)<(m,\ell)$.
\end{itemize}
If $d$ has an ordinal bound then $D_{m,\ell}(d)$ has an ordinal bound.
\end{theorem}
\begin{proof}
  We proceed by induction on the ordinal bound $o^d_{\langle\rangle}$. If the last inference rule of $d$ is anything other than $\Omega^\flat$, Cut$\Omega^\flat$ then the claim follows immediately from the inductive hypothesis. (We will not encounter Read rules at complexity $(m,\ell)$, as we will see below.)

  If the last inference is a Cut$\Omega^\flat$ with complexity $(m,\ell)$, we can use the fact that in this case, $D_{m,\ell}(d)$ is $D_{m,\ell}(\overline{d_{\langle\bot\rangle}^\uparrow}(D_{m,\ell}(d_{\langle\top\rangle})))$. The inductive hypothesis gives us an ordinal bound on $D_{m,\ell}(d_{\langle\top\rangle})$, and therefore the ordinal bound on $\overline{d_{\langle\bot\rangle}^{\uparrow}}(D_{m,\ell}(d_{\langle\top\rangle}))$ is less than $o^d$, so the inductive hypothesis applies again.

  Suppose that last inference is an $\Omega^\flat$ inference. For any $d'$ in $PA^{\infty,N,m',\ell'}_{2,<0}$, we have a deduction $\overline{d_{\langle\bot\rangle}^{\uparrow}}(d')$, and if $d'$ has an ordinal bound then $\overline{d_{\langle\bot\rangle}^{\uparrow}}(d')$ has an ordinal bound below the ordinal bound of $d$, so the inductive hypothesis applies to $\overline{d_{\langle\bot\rangle}}(d')$.

  Taking the observation that $D_{m,\ell}(d_{\langle\bot\rangle})$ is $d'\mapsto D_{m,\ell}(\overline{d_{\langle\bot\rangle}^\uparrow}(d'))$, we can assign an ordinal bound to $D_{m,\ell}(d_{\langle\bot\rangle})$: for any $\sigma$ we assign to $\sigma$ the bound mapping $o^*$ to the supremum, over all $d'$ with bound $o^*$, of the bound in $D_{m,\ell}(\overline{d_{\langle\bot\rangle}^{\uparrow}}(d'))=\overline{D_{m,\ell}(d_{\langle\bot\rangle})^{\uparrow}}(d')$ of those $\tau$ with $h(\tau)=\sigma$.

  The Cut$\Omega^\flat$ case is similar, taking into account the bound on $d_{\langle\top\rangle}$ as well.
\end{proof}

\begin{theorem}
  If $PA_2\vdash\phi$ where $\phi$ is a $\Sigma_1$ formula then $PA_{2,<0}\vdash\phi$. In particular, $PA_2$ is consistent.
\end{theorem}
\begin{proof}
  Suppose $PA_2\vdash\phi$. Theorem \ref{thm:embedding} gives us $n,L,r$ and a proof-tree $d^\infty$ in $PA^{\infty,N,L}_{2,<r}$ so that $\Gamma(d^\infty)\subseteq\{\phi\}$. By Lemma \ref{thm:embedding_bounds} there are ordinal bounds on $d^\infty$.

  Consider the deduction $d'=\mathrm{Reduce}_0^{\uparrow}(\cdots\mathrm{Reduce}_{r-1}^\uparrow(d^\infty)\cdots)$; by Theorems \ref{thm:cut_elim_bounds} and \ref{thm:bound_lifting}, $d'$ has an ordinal bound as well.

  Finally, we let $d''=D_{0,0}(\cdots D_{n,L}(d')\cdots)$. By Theorem \ref{ref:collapsing_bounds}, $d''$ has ordinal bounds as well. Since $d''$ is a deduction in $PA^{\infty,N,L,-}_{2,<0}$, an ordinal bound on $d''$ is simply an assignment of ordinals: $d''$ is well-founded.

  Since $PA^{\infty,N,L,-}_{2,<0}$ satisfies the subformula property, the only rules that can appear in $d''$ are True, Ax, I$\wedge$, I$\vee$, I$\exists$, and Rep. Removing Rep rules leaves us with a deduction in $PA_{2,<0}$.
\end{proof}

\printbibliography
\end{document}